\documentclass[a4paper,11pt, twoside]{amsart}
\usepackage[utf8]{inputenc}
\usepackage[inner=2.5cm,outer=2.5cm,top=3cm,headsep=1cm, footskip=1cm,bottom=3cm]{geometry}
\usepackage[english]{babel}
\usepackage{amsfonts,latexsym,rawfonts,amsmath,amssymb,amsthm,amscd}
\usepackage{fancyhdr}
\usepackage{enumerate}
\usepackage{stackrel}

\usepackage{hyperref}
\usepackage{mdwlist}
\usepackage{array} 
\input xy
\xyoption{all}

\hypersetup{
    colorlinks,%
    citecolor=black,%
    filecolor=black,%
    linkcolor=black,%
    urlcolor=black
}

\makeatletter
\let\uppercasenonmath\@gobble
\makeatother

\usepackage{titlesec}
\titleformat{\section}{\bfseries\center}{\thesection.}{0.4em}{}
\titleformat{\subsection}{\vspace{.08cm}\bfseries}{\thesubsection.}{0.4em}{}

\newtheorem{prop}{Proposition}[section]
\newtheorem{teo}[prop]{Theorem}
\newtheorem{lem}[prop]{Lemma}
\newtheorem{cor}[prop]{Corollary}

\theoremstyle{definition}

\newtheorem{defi}[prop]{Definition}

\theoremstyle{theorem}

\newcommand{\mat}[1]{\left(\begin{matrix}#1\end{matrix}\right)}

\newcommand{\Gas}[1]{\Gamma_{#1}}
\newcommand{\Gr}[2]{\mathbb{G}(#1,#2)}
\newcommand{\Grl}[2]{G(#1,#2)}
\newcommand{\Gs}[2]{\mathbb{G}{(#1,\Sp^{#2})}}
\newcommand{\Gh}[2]{\mathbb{G}{(#1,\HH^{#2})}}
\newcommand{\Gaf}[2]{\mathbb{G}{(#1,\EE^{#2})}}
\newcommand{\Gah}[1]{\Gamma_{#1}}

\newcommand{\vs}{\vspace{0.2cm}}

\def\Ker{\mathrm{Ker }}
\def\Euc{\mathrm{Euc}}
\def\End{\mathrm{\End}}
\newcommand{\eps}{\varepsilon}

\newcommand{\BB}{\mathbb{B}}
\newcommand{\CC}{\mathbb{C}}
\newcommand{\EE}{\mathbb{E}}
\newcommand{\HH}{\mathbb{H}}

\newcommand{\PP}{\mathbb{P}}

\newcommand{\RR}{\mathbb{R}}
\newcommand{\Sp}{\mathbb{S}}

\newcommand{\Mm}{\mathcal{M}}

\newcommand{\kk}{\mathbf{k}}

\title[Classification of isometries and invariant subspaces]
{{\large Classification of isometries of spaces of constant curvature\\ and invariant subspaces}}

\author{Joana Cirici}
\email{jcirici@math.fu-berlin.de}
\address[J. Cirici]{
Fachbereich Mathematik und Informatik\\
Freie Universit\"{a}t Berlin\\  Arnimallee 3\\ 
14195 Berlin}

\thanks{Financial support from the Marie Curie Action through PCOFUND-GA-2010-267228.
Partially supported by the Spanish Ministry of Economy and Competitiveness under project MTM 2009-09557 and by the 
DFG through project SFB 647.}

\begin{document}
\maketitle
\begin{abstract}
We study the varieties of invariant totally geodesic submanifolds of isometries
of the spherical, Euclidean and hyperbolic spaces
in each finite dimension.
We show that the dimensions of the connected
components of these varieties determine the orbit type (or the $z$-class) of the isometry.
For this purpose, we introduce the Segre symbol of an isometry, a discrete invariant
encoding the structure of its normal form, which parametrizes $z$-classes.
We then provide a description of the isomorphism type of the varieties of invariant subspaces
in terms of the Segre symbol.
\end{abstract}

\section{Introduction}
Our objective is to study the isometries of spaces of constant curvature.
We will denote by $\mathbf{I}(\mathbb{M}^n)$ the isometry group of $\mathbb{M}^n$,
where $\mathbb{M}^n$ denotes one of the simply connected space forms of finite dimension $n$:
the sphere $\Sp^n$, the Euclidean space $\EE^n$ or the hyperbolic space $\HH^n$.

A very first and well known classification of elements of $\mathbf{I}(\mathbb{M}^n)$
is based on their fixed point behavior.
The \textit{translation length} of an isometry $f$ is the infimum of the distances that points are moved.
If $f$ has a fixed point it is called \textit{elliptic}. If the translation length 
is positive then $f$ is called \textit{hyperbolic}.  If $f$ does not
fix a point but its translation length is zero, then it is called \textit{parabolic}.
We remark that the latter case can only occur in hyperbolic spaces (see for example \cite{Gr}). 
This is a very coarse classification, only useful up to a certain extent,
and in some sense analogous to the distinction between direct and indirect isometries.

At the other extreme, we can consider the classification of isometries by conjugacy classes. This gives 
infinitely many classes, each one represented by a normal form.
A natural step is to identify those classes 
differing in the sets of eigenvalues but
having the same normal form structure. This leads 
to what we call the \textit{Segre decomposition}: a decomposition of $\mathbf{I}(\mathbb{M}^n)$ into finitely many classes,
each an uncountable union of conjugacy classes, parametrized by a discrete invariant.

The \textit{Segre symbol} of a linear endomorphism of a finite dimensional complex vector space is a
sequence of positive integers encoding the number of distinct eigenvalues, together with the sizes of
the Jordan blocks corresponding to each eigenvalue of a Jordan normal form of the endomorphism. This invariant has been studied by many authors
for classification purposes, such as the treatment of collineations and pencils of quadrics developed in $\S$.VIII of \cite{HP}
or Petrov's classification of gravitational fields appearing in $\S$.3 of \cite{Pe}, in which the Segre symbol 
of the curvature tensor allows to stratify Einstein spaces.

Arnold \cite{Ar} studied the conjugation action of $\mathrm{GL}(n,\CC)$ on the space $\Mm_n(\CC)$ of complex square matrices of size $n$ and 
suggested that the partition into Segre classes (defined by those matrices with a given Segre symbol) makes $\Mm_n(\CC)$ into a stratified space.
Gibson \cite{Gi} proved this statement and showed that the stratification is
Whitney-regular. Furthermore, in this case the 
partition according to the Segre symbol coincides with the stratification by orbit types (see \cite{Br}).

Consider a Lie group $G$ acting on a manifold $M$. Then two elements of $M$ are said to have the same \textit{orbit type} if their
orbits are $G$-equivariantly isomorphic. This is equivalent to the condition that their isotropy subgroups
are conjugate. In the case in which a group acts on itself by conjugation, then the orbit types are also called
\textit{centralizer classes} or \textit{z-classes}.
This notion was introduced by Kulkarni in \cite{Ku1} in his study of “dynamical types”.
The classification by $z$-classes using characteristic and minimal polynomials has been developed in many
particular cases, such as for the general linear and affine groups \cite{Ku2},
the anisotropic groups of type $G_2$ defined over a field \cite{Sin},
or the groups of isometries of the real, 
quaternionic and complex hyperbolic spaces 
\cite{GK}, \cite{Go3}, \cite{GP}.

In the present paper we define the Segre symbol of isometries of $\mathbb{M}^n$ using normal forms.
Our definition is a natural adaptation from the original Segre symbol of a linear endomorphism,
to isometries of spaces of constant curvature, and involves the distinction between elliptic,
hyperbolic and parabolic isometries.
We then study the $z$-classes of $\mathbf{I}(\mathbb{M}^n)$ and show that
the Segre decomposition coincides with the decomposition by $z$-classes.
In particular, the number of $z$-classes is finite, and we can count them 
by simple combinatoric arguments.
Each $z$-class defines a stratum of $\mathbf{I}(\mathbb{M}^n)$. 
We compute its dimension in terms of its Segre symbol. 
Our approach differs from that of \cite{Ku2} and \cite{GK}, in that we strongly use the existence of normal forms. In particular,
to study the $z$-classes of isometries we just describe the centralizer group for each type of normal form.

The Segre decomposition (or equivalently, the decomposition by $z$-classes) 
provides a satisfactory classification of isometries: the number of classes is finite and their dynamical behavior
is encoded in the shape of each normal form. Each Segre symbol defines a stratum of a given dimension.

The motivation of this paper is to relate the Segre decomposition of isometries with the ad hoc classifications
that exist in the literature for isometries of lower dimensional spaces, allowing to generalize 
those geometric treatments to arbitrary dimensions.

Isometries of the two- and three-dimensional Euclidean spaces are usually classified
according to the sets of fixed points and invariant lines (see for example \cite{Cox}, \cite{Mar} or \cite{RT}).
For the Euclidean plane, this classification coincides with the Segre decomposition, but for the three-dimensional Euclidean space,
we find that there are two Segre symbols that are merged into the same class.
Indeed, the sets of fixed points and invariant lines allow to
distinguish a rotation from an axial symmetry (a rotation of angle $\pi$) in $\EE^3$. However, when
composed with a translation along the rotation axis, 
both isometries have no fixed points and a single fixed line. As a consequence, in the above references,
the two isometries are considered in a single class, while their $z$-classes differ.
In order to solve this discordance, we observe that it suffices to also consider
 the sets of invariant planes.

The above example motivates the study of higher dimensional invariant subspaces of $\EE^n$,
in order to classify its isometries. A generalization of this study to isometries of the curved spaces
$\Sp^n$ and $\HH^n$ leads to the study of totally geodesic
invariant subspaces.

We recall that a Riemannian submanifold $N$ of a Riemannian manifold $M$ is called \textit{totally geodesic}
if all geodesics in $N$ are also geodesics in $M$. For example, each closed geodesic in a Riemannian
manifold defines a 1-dimensional compact totally geodesic submanifold.
The generalized Grassmannian $\Gr{k}{\mathbb{M}^n}$ is defined as the set of closed totally
geodesic submanifolds of $\mathbb{M}^n$ isometric to $\mathbb{M}^k$, where $0\leq k\leq n$.
It is a homogeneous space of dimension $(k+1)(n-k)$.
Given an isometry $f\in \mathbf{I}(\mathbb{M}^n)$ we denote by $\Gamma_f(k)$ the 
closed subset of $\Gr{k}{\mathbb{M}^n}$
defined by those submanifolds that are $f$-invariant.

In this paper we describe the sets $\Gamma_f(k)$ and relate them to the Segre symbol:
we show that for all $0\leq k\leq n$, the set $\Gamma_f(k)$ is a closed smooth submanifold of $\Gr{k}{\mathbb{M}^n}$, and that
given isometries $f,g\in\mathbf{I}(\mathbb{M}^n)$, then we have that $\Gamma_f(k)\cong \Gamma_g(k)$
for all $0\leq k\leq n$ if and only if $f$ and $g$ have the same Segre symbol (and hence they are in the same $z$-class).
In fact, we prove a stronger result: the dimensions of 
the connected components of $\Gamma_f(k)$, for $0\leq k\leq 4$, determine the Segre symbol (and thus the $z$-class) of the isometry $f$.

The fact that the varieties of invariant subspaces determine the Segre symbol of isometries holds
in other contexts of similar nature. For instance,
we have been able to prove similar results for linear and affine endomorphisms of complex vector spaces. 
In this case the varieties of invariant subspaces are not smooth, but admit a stratification into smooth 
varieties according to restricted Segre symbols (see \cite{Sh}).
This theory involves a further study of partitions and Young diagrams, and will be developed elsewhere.
\\

We next explain the contents of each section.
Although normal forms of isometries are known, precise statements and proofs for
Euclidean and hyperbolic isometries are difficult to find in the literature. Normal forms of $\Euc(n)$
can be found in \cite{RT} with a slight different notion of normal form than the one presented here.  
Isometries of the hyperbolic space are often
studied by means of the M\"{o}bius group, by considering the upper-half plane model of the hyperbolic
space (see for example \cite{Ah1}, \cite{Ah2}, \cite{Go2}, \cite{Ra}). This may possibly be a cause for the lack of references concerning normal forms of
the Lorentz group $\mathrm{O}(1,n)$. In the classical papers \cite{Ab} and \cite{Gr} such normal forms are listed, 
but apparently these are not well enough known.
For instance, in Theorem 6.9 of \cite{Ba}, the list of normal forms is incomplete since parabolic isometries
are missing. Since we use normal forms as a key step in the study of
orbit types and invariant subspaces, we present a survey of these forms in Section 2 with elementary proofs purely based on
linear algebra arguments. 

Section 3 is devoted to the study of $z$-classes of isometries. We compute the centralizers of normal forms
via a product formula accounting for the decomposition of normal forms into primary components.
The main results of this section are Theorems $\ref{otort}$, $\ref{zclasseuc}$ and $\ref{zclasshyp}$, where we relate the $z$-classes with the Segre symbol
of isometries of spherical, Euclidean and hyperbolic spaces respectively. As a direct application, we
compute the dimensions of centralizer groups and of each Segre stratum, in terms of the Segre symbol.

In Section 4 we study the varieties of invariant totally geodesic
subspaces of an isometry. We provide a description of these varieties in terms of the Segre symbol, and
show that their connected components are products of generalized Grassmannians. We then prove the
main result of this paper in each case (see Theorems $\ref{main1}$, $\ref{main2}$ and $\ref{main3}$),
relating the Segre symbol of isometries with the varieties of invariant subspaces.
We detail the classification of isometries for spaces in 1, 2 and 3 dimensions
(see Tables $(1)-(9)$), which we believe, should be the standard classification presented in basic linear geometry courses.
The tables show a normal form type for each Segre symbol, the dimensions of the strata
and the varieties of invariant subspaces in each case.

\section{Normal forms and Segre symbol}
In this expository section we review the normal forms of isometries of simply connected space forms $\mathbb{M}^n$ in each finite dimension.
We then define the Segre symbol of an isometry of $\mathbb{M}^n$.
A Segre stratum in $\mathbf{I}(\mathbb{M}^n)$
consists of all isometries having a given Segre symbol. We define the Segre decomposition to be the
partition of $\mathbf{I}(\mathbb{M}^n)$ into Segre strata.
There are finitely many strata, each a collection of conjugacy classes having the same discrete
invariants but different eigenvalues.
We compute the number of Segre strata in $\mathbf{I}(\mathbb{M}^n)$ using combinatoric arguments.

\subsection{Spherical isometries}
Consider the standard model of the $n$-dimensional spherical space $\Sp^n=\{x\in \RR^{n+1}; |x|=1\}$.
The group $\mathbf{I}(\Sp^n)$ of isometries of $\Sp^n$ is isomorphic to the real orthogonal group $\mathrm{O}(n+1)$.

\begin{defi}
An \textit{orthogonal normal form} is a block diagonal matrix whose blocks are
$\pm \mathbb{I}_k$ or $R_\theta$, where
$\mathbb{I}_k$ denotes the identity matrix of size $k\times k$ and 
$$R_\theta=\mat{
\cos \theta &\text{-}\sin \theta\\
\sin \theta &\cos \theta\\
}\,;\, 0<\theta<\pi.$$
\end{defi}

Denote by $A\oplus B$ the block diagonal matrix whose blocks are the direct summands $A$ and $B$.
Likewise, denote by $A^{\oplus n}$ the block diagonal matrix consisting of $n$ blocks equal to $A$.

\begin{teo}\label{orthogonalforms}
Every element of $\mathrm{O}(n)$ is conjugate to an orthogonal normal form, unique up to block permutation.
\end{teo}
\begin{proof}
Let $A$ be a real square matrix and consider it as a matrix with values in $\CC$.
Let $z=x+iy$ be an eigenvector corresponding to some eigenvalue $\lambda=a+ib$.
Then $Ax=ax-by$ and $Ay=ay+bx$. Hence the subspace 
$W=Sp\{x,y\}\subset \RR^2$ is $A$-invariant.

Assume that $A\in \mathrm{O}(n)$ is an orthogonal matrix. 
The above argument implies that there exists 
an $A$-invariant subspace $W\subset\RR^n$ of dimension 1 or 2.
Since its orthogonal complement 
$W^\bot\subset\RR^n$ is also $A$-invariant and of dimension $<n$ we
may apply induction to see that $\RR^n=W_1\oplus \cdots\oplus W_k$ 
decomposes into a direct sum of $A$-invariant 
subspaces with $\dim W_i\in\{1,2\}$.
Therefore $A$ is conjugate to a block diagonal
matrix $B=B_1\oplus\cdots \oplus B_k$ where
$B_i$ is defined by restriction of $B$ to $W_i$.
To conclude the proof it suffices to note that $\mathrm{O}(1)=\{\pm 1\}$ and 
that every element of
$\mathrm{O}(2)$ with no real eigenvalues is of the form $R_\theta$ for some $0<\theta<\pi$.
\end{proof}

\begin{defi}
Let $f\in \mathbf{I}(\Sp^n)$. By Theorem $\ref{orthogonalforms}$, $f$ has an orthogonal normal form
$$R_{\theta_1}^{\oplus n_1}\oplus\cdots\oplus R_{\theta_s}^{\oplus n_s}\oplus \pm(\mathbb{I}_{m_1}\oplus \text{-}\mathbb{I}_{m_2})$$
where $\theta_i\neq \theta_j$ for $i\neq j$, $n_1\geq \cdots\geq n_s\geq 0$ and $m_1\geq m_2\geq 0$.
The \textit{Segre symbol of $f$} is
$${\tilde\sigma}_f=[(n_1\bar{n}_1), \cdots,(n_s\bar{n}_s),m_1,m_2]$$
where entries that are zero are omitted from the notation.
\end{defi}
For example, if $R_\theta\oplus \mathbb{I}_2$ is a normal form of $f\in\mathbf{I}(\Sp^3)$, then $\tilde\sigma_f=[(1\bar 1),2]$.
The notation $(n\bar n)$ indicates that over $\CC$,
there is a pair of conjugate complex eigenvalues of multiplicity $n$.

\begin{prop}\label{numesferiques}
The number of Segre classes $s(n)$ of $\mathbf{I}(\Sp^n)$ is given by
$$s(n)=\sum_{j=0}^{[{{n+1}\over{2}}]} p(j)\left(\left[{{{n+1}}\over{2}}\right]-j+1\right),$$
where $p(k)$ denotes the number of partitions of $k$ and $[k]$ denotes the integer part of $k$.
\end{prop}
\begin{proof}
The sum is taken over the number $j$ of blocks of type $R_\theta$. Hence the partition number $p(j)$ accounts for the
combinations $n_1\geq \cdots \geq n_s$ with $n_1+\cdots+n_s=j$.
The number of partitions of an integer $k$ into one or two parts is equal to $[{k\over 2}]+1$.
Taking $k=(n+1)-2j$ we get the formula between brackets, which accounts for the pairs $m_1\geq m_2$ with $m_1+m_2=k$.
\end{proof}

\subsection{Euclidean isometries}
Consider the standard affine space $\EE^n$ as the subspace of $V=\RR^{n+1}$ given by $\EE^n=\{x_{n+1}=1\}$.
The subspace of $V$ given by $V_0=\{x_{n+1}=0\}$ is called the \textit{vector space associated to $\EE^n$}.
We will assume that $V_0$ comes with the standard Euclidean metric.
\begin{defi}
A map $f:\EE^n\to \EE^n$ is an \textit{isometry} if it is induced by a linear map
$\varphi:V\to V$ satisfying $\varphi(\EE^n)\subset \EE^n$ and for which
the restriction $\varphi|_{V_0}:V_0\to V_0$ is an orthogonal map with respect
to the metric of $V_0$.
\end{defi}
The group $\mathbf{I}(\EE^n)$ of isometries of $\EE^n$ is isomorphic to
the Euclidean group $\Euc(n)=\mathrm{O}(n)\ltimes \RR^n$.

Recall that an \textit{Euclidean reference} $(\{e^i\}_{i=1}^n;p)$ of $\EE^n$ is an orthonormal 
basis $\{e^i\}_{i=1}^n$ of $V_0$ and a point $p\in \EE^n$. It
induces a basis $\{e^{i}\}_{i=1}^{n+1}$ of $V$, with $e^{n+1}=p$.

Let $f\in\mathbf{I}(\EE^n)$. The \textit{matrix of $f$} in an Euclidean reference
$(\{e^i\}_{i=1}^n;p)$ is the matrix of $\varphi$ in the induced basis of $V$. 
Since $\varphi(V_0)\subset V_0$,
the last row of such a matrix is $(0,\cdots,0,1)$.

\begin{defi}
An \textit{Euclidean normal form} is a matrix of one of the following types:
\begin{enumerate}[1)] 
 \item \textit{Elliptic}: $A\oplus 1$, where $A\in \mathrm{O}(n)$ is an orthogonal normal form.
\item  \textit{Hyperbolic}: $A\oplus
\left(\begin{smallmatrix}
1&a\\
0&1
\end{smallmatrix}\right)$, where ${A}\in \mathrm{O}(n-1)$ is an orthogonal normal form and $a\in \RR$ is a positive number.
\end{enumerate}
\end{defi}

\begin{teo}\label{teoeuclidi}
Let $f\in\mathbf{I}(\EE^n)$. There exists an Euclidean reference of $\EE^n$ such
that the matrix of $f$ in this reference is an Euclidean normal form.
\end{teo}
\begin{proof}
Define a linear form $w:V\to \RR$ by the projection to the last component. 
Then $V_0=\Ker(w)$. An isometry
$f$ is determined by a linear map $\varphi:V\to V$ such that 
$w\circ \varphi=w$ and $\varphi_0=\varphi|_{V_0}$ is orthogonal. 
Therefore it suffices to show that there exists a Jordan
basis $\{u^i\}_{i=1}^{n+1}$ of $V$ with respect to $\varphi$, 
satisfying $w(u^{n+1})=1$ and $w(u^i)=0$ for all $i\leq n$, 
and such that $\{u^1,\cdots,u^n\}$ is a Jordan orthonormal basis of $V_0$ with respect to $\varphi_0$.
Let $V=V_R\oplus V_1$ be the decomposition of $V$ into $\varphi$-invariant subspaces, where
$V_{1}$ denotes the generalized eigenspace of $V$ corresponding to the eigenvalue $1$.
Then $V_R=\mathrm{Im}(\varphi-I)^s$ for some $s\geq 0$. Since
$w\circ(\varphi-I)=0$, we have $V_R\subset V_0$. Consequently, 
there is a Jordan orthonormal basis of $V_R$ with respect to $\varphi_0|_{V_R}$. 
Moreover, there is an orthogonal decomposition $V_0=V_R\oplus (V_1\cap V_0)$.
Therefore we can assume that $V=V_1$, so that $\varphi$ is a unipotent linear map. 
Since $\varphi_0$ is orthogonal, it is the identity matrix.
Then $f$ is either the identity or a translation. In the first case the proof is 
complete. Assume that $f$ is a translation and let $p=(0,\cdots,0,1)$. 
Then $p-f(p)\neq 0$ and we define $u^n=(f(p)-p)/||f(p)-p||$. 
We complete $u^n$ to an orthonormal basis of $V_0$.
Then $\{u^1,\cdots,u^n,p\}$ is a basis of $V$ satisfying the desired conditions.
The matrix of $f$ in the reference $(u^1,\cdots,u^n;p)$ is $\mathbb{I}_{n-1}\oplus\left(\begin{smallmatrix}
1&a\\
0&1
 \end{smallmatrix}\right)$, where $a=||f(p)-p||$ is a positive real number.
\end{proof}

We next define the Segre symbol of an Euclidean isometry.
The main difference with respect to the orthogonal case is that now we distinguish the elliptic and hyperbolic cases, as well as the blocks of
eigenvalue 1 from the remaining blocks. This distinction will be justified in Section $\ref{seccioorbittypes}$.

\begin{defi}
Let $f\in \mathbf{I}(\EE^n)$. By Theorem $\ref{teoeuclidi}$, $f$ has a normal form $A=A_{R}\oplus A_1$
where $A_R\in \mathrm{O}(n-r)$ is an orthogonal normal form with no eigenvalues equal to 1 and $A_1\in\Euc(r)$ is a unipotent Euclidean normal form
satisfying one of the following:
\begin{enumerate}[1)]
\item \label{eli} \textit{Elliptic}: $A_1=\mathbb{I}_{r+1}$.
\item \label{hipe} \textit{Hyperbolic}:
$A_1=\mathbb{I}_{r-1}\oplus\left(\begin{smallmatrix}1&a\\0&1\end{smallmatrix}\right)$,
where $a>0$.
\end{enumerate}
The \textit{Segre symbol of $f$} is defined by $\sigma_f=[\tilde\sigma;r;t]$, 
where $t\in\{e,h\}$ denotes the type of isometry (elliptic or hyperbolic) and
$\tilde\sigma$ is the Segre symbol of $A_R$.
\end{defi}

\begin{prop}\label{segreclasseseuclidean}
The number of Segre classes $e(n)$ of $\mathbf{I}(\EE^n)$ is given by
$$e(n)=\sum_{i=n-1}^n\left(\sum_{j=0}^{[i/2]} p(j)(i-2j+1)\right),$$
where $p(k)$ is the number of partitions of $k$, and $[k]$ denotes the integer part of $k$.
\end{prop}
\begin{proof}
Let $e_{e}(n)$, $e_h(n)$ denote the number of Segre classes of elliptic and hyperbolic Euclidean isometries respectively.  
By definition of Euclidean normal forms $e_h(n)=e_e(n-1)$. Therefore
we have $e(n)=e_e(n)+e_h(n)=e_e(n)+e_e(n-1)$.
It remains to show that
$$e_e(n)=\sum_{j=0}^{[n/2]} p(j)(n-2j+1).$$
Indeed, the sum is taken over the number $j$ of blocks of type $R_\theta$. The partition number $p(j)$ accounts for the
combinations of blocks of such type, while $(n-2j+1)$ accounts for the possible combinations $-\mathbb{I}_{k}\oplus \mathbb{I}_{r}$,
with $k+r=n-2j$.
\end{proof}

\subsection{Preliminaries on Lorentz spaces and Lorentz groups}
For the rest of this section we fix a real vector space $V$ of dimension $n+1$ together with a
non-degenerate symmetric bilinear form $Q$ over $V$ of signature $(1,n)$.
The pair $(V,Q)$ is called a \textit{Lorentz space}. We will denote by $Q(x):=Q(x,x)$ the associated quadratic form.

\begin{defi}
A vector $v\in V$ is called \textit{space-like} if $Q(v)>0$, it is called \textit{time-like} if $Q(v)<0$, and it is called \textit{light-like}
if $Q(v)=0$. 

A linear subspace $U\subset V$ is said to be \textit{time-like} if it has a time-like vector, 
\textit{space-like} if every nonzero vector in $U$ is space-like, and \textit{light-like} otherwise.
\end{defi}

\begin{lem}[\cite{Ba}, Prop. 6.2]\label{condics} If two linearly independent vectors of $V$ are orthogonal then at least one of them is space-like.
\end{lem}
\begin{proof}Assume that $Q(u)\leq 0$.
We can write $Q(u,v)=-u_0v_0+\sum u_iv_i=0$.
Since $u_0\neq 0$ we have $v_0=\sum u_iv_i/u_0$.
Then $Q(v)\geq 0$, with equality only if $Q(u)=0$. If $u$ and $v$ are linearly independent then $v$ is space-like.
\end{proof}

\begin{defi}
An \textit{isometry of $(V,Q)$} is a linear endomorphism $T\in \mathrm{End}(V)$ such that $Q(Tu,Tv)=Q(u,v)$ for all $u,v\in V$. 
\end{defi}
Denote by $\mathrm{O}(Q)$ the group of isometries of $(V,Q)$.

\begin{lem}\label{spacetime}
Let $T\in \mathrm{O}(Q)$. If $W$ is a space-like $T$-invariant subspace of $V$ then $W^{\bot}$ is a $T$-invariant Lorentz space and $V=W\oplus W^{\bot}$.
\end{lem}
\begin{proof}
Since $W$ is space-like we have $W\cap W^{\bot}=0$ and $W\oplus W^{\bot}=V$. Since $T$ is an isometry, $W^\bot$ is $T$-invariant. Indeed, if
$u\in W^\bot$ and $v\in W$, then $Q(Tu,v)=Q(u,T^{-1}v)=0$. Therefore $Tu\in W^{\bot}$. Since $V$ is time-like and $W$ is space-like, $W^{\bot}$
must be time-like, and hence a Lorentz space.
\end{proof}
Denote by $(-)^c$ the scalar-extension functor from $\RR$ to $\CC$.
If $T$ is an isometry of $(V,Q)$ then $T^c$ is an isometry of $(V^c,Q^c)$. 
Moreover, conjugacy in $\CC$ induces an $\RR$-automorphism of $V^c$, which we denote by $v\mapsto \bar v$, as it is customary.
We have $\RR$-linear maps
$\mathcal{R},\mathcal{I}:V^c\to V\otimes \RR\cong V,$
defined by $$\mathcal{R}(v)={{1}\over{2}}(v+\bar v)\text{ and }\mathcal{I}(v)={{1}\over{2}}(v-\bar v).$$

We next study the orthogonality of the generalized eigenspaces of $T^c$.
The following result is possibly well known. We give an elementary proof.
\begin{lem}\label{le2}
Let $T\in \mathrm{O}(Q)$ and let $\lambda,\mu\in \CC$ such that $\lambda\mu\neq 1$.
The generalized eigenspaces $V^c_\lambda$ and $V^c_\mu$ of $V^c$ with respect to $T^c$,
corresponding to the eigenvalues $\lambda$ and $\mu$ respectively, are orthogonal: $Q^c(V_\lambda^c,V_\mu^c)=0$.
\end{lem}
\begin{proof}
Throughout this proof we will omit the superscript $c$. Let $F_\lambda^i=\Ker (T-\lambda)^i$, and 
$F_\mu^j=\Ker(T-\mu)^j$. We want to show that $Q(u,v)=0$ for all $u\in F_\lambda^i$ and $v\in F_\mu^j$. 
We will prove it by induction over $i+j$. If $i+j\leq 1$ it is trivial.
The first non-trivial case is $i=j=1$. Let $u\in F_\lambda^1$ and $v\in F_\mu^1$.
Then $Tu=\lambda u$, and $Tv=\mu v$. Therefore
$Q(u,v)=Q(Tu,Tv)=\lambda\mu Q(u,v)=0.$
Assume that $i+j\geq 2$. Let $u_1\in F_\lambda^i$ and $u_2\in F_\mu^j$, and let
$w_1=(T-\lambda)u_1$, $w_2=(T-\mu)u_2$. Since $w_1\in F_\lambda^{i-1}$ and $w_2\in F_\mu^{j-1}$, 
by the induction hypothesis,
$Q(w_1,w_2)=Q(w_1,u_2)=Q(u_1,w_2)=0.$
Therefore
$Q(u_1,u_2)=Q(Tu_1,Tu_2)=Q(w_1+\lambda u_1,w_2+\mu u_2)=\lambda\mu Q(u_1,u_2).$
Since $\lambda\mu\neq 1$ we have $Q(u_1,u_2)=0$.
\end{proof}

\begin{defi}
The \textit{space-time decomposition} of $V$ with respect to an isometry $T\in \mathrm{O}(Q)$ is the decomposition
$V=V_t\oplus V_s$, where $V_s$ is the direct sum of all space-like generalized eigenspaces 
of $V$ with respect to $T$ and $V_t$ is the direct sum of the remaining generalized eigenspaces.
We will call $T_s=T|_{V_s}$ the \textit{spatial component of $T$} and $T_t=T|_{V_t}$,
the \textit{temporal component of $T$}.
\end{defi}
The following properties are a consequence of Lemmas $\ref{spacetime}$ and $\ref{le2}$.
\begin{prop} Let $T\in \mathrm{O}(Q)$ and let $V=V_t\oplus V_s$ be the space-time decomposition 
of $V$ with respect to $T$. 
\begin{enumerate}[i)]
\item The spaces $V_t$ and $V_s$ are $T$-invariant coprime orthogonal subspaces.
\item The space $V_s$ is Euclidean and the space $V_t$ is Lorentzian.
\item The map $T_s$ is an orthogonal isometry, while $T_t$ is a Lorentz isometry.
\end{enumerate}
\end{prop}
We next study the eigenvalues of a Lorentz isometry. Since $T_s$ is orthogonal,
we will only deal with the temporal component $T_t$.
\begin{lem}\label{le3j}
Let $T\in \mathrm{O}(Q)$ be a temporal isometry. Let $\lambda\in \CC$ be an
eigenvalue of $T^c$. Then 
$\lambda$ is real, and $\lambda^{-1}$ is also an eigenvalue of $T$.
\end{lem}
\begin{proof}
Assume that $\lambda\in \CC\setminus \RR$ is an eigenvalue of $T^c$ and denote by
$V^c_\lambda$ the generalized eigenspace of $V^c$ of eigenvalue $\lambda$. 
Then $\overline{V}^c_\lambda=V^c_{\bar \lambda}$ is the generalized eigenspace of 
 $\bar\lambda$. Since $\lambda^2,\bar\lambda^2\neq 1$, by Lemma $\ref{le2}$ we
have $Q|_{V_\lambda^c}=Q|_{\bar V_\lambda^c}=0$. Denote by $V_{\lambda\bar\lambda}=\mathcal{R}(V_\lambda^c)$ 
the generalized eigenspace of $V$ corresponding to $p(x)=x^2-2\mathcal{R}(\lambda)x+|\lambda|^2$.
We next show that $V_{\lambda\bar\lambda}$ is space-like. Indeed, if
$v_1\in V_{\lambda\bar\lambda}$, there exists $w\in V_\lambda^c$ such that $v_1=\mathcal{R}(w)$.
Let $v_2=\mathcal{I}(w)$. Then $v_2\in V_{\lambda\bar\lambda}$ and
$$Q(v_1)=Q(v_2)={1\over 2}Q(w,\bar w) \text{ and }Q(v_1,v_2)=Q(\mathcal{R}(w),\mathcal{I}(w))=0.$$
Since $v_1,v_2$ are linearly independent, by Lemma $\ref{condics}$, $Q(v_1),Q(v_2)>0$.
Therefore $V_{\lambda\bar\lambda}$ is space-like. This is a contradiction, since $V=V_t$ is the direct sum of all
non space-like generalized eigenspaces. Hence $\lambda$ must be real.

Let $\lambda$ be a real eigenvalue of $T$ and let $A$ be 
the matrix of $T$ in an orthonormal basis. If $A^tJA=J$ then $J^{-1}A^tJ=A^{-1}$.
Hence $\lambda^{-1}$ is an eigenvalue of $T$.
\end{proof}
\begin{cor}\label{pr4j}
Let $T\in \mathrm{O}(Q)$. Then either $V_t=V_{\pm 1}$ or
$V_t=V_\lambda\oplus V_{\lambda^{-1}}$, where $\lambda\neq \pm 1$ is real.
\end{cor}
\begin{proof}
By Lemma $\ref{le3j}$ we have $V_t$ decomposes into a direct sum of generalized eigenspaces of the type $V_{\pm 1}$ and $V_{\lambda}\oplus V_{\lambda^{-1}}$, 
where $\lambda\neq \pm 1$ is a real eigenvalue. By Lemma $\ref{le2}$, these subspaces are pairwise orthogonal. 
Therefore by Lemma
$\ref{condics}$, all except one are space-like. Since $V_t$ contains no 
space-like generalized eigenspaces, it follows that $V_t=V_{\pm 1}$ or $V_t=V_\lambda\oplus V_{\lambda^{-1}}$.
\end{proof}

We next study the minimal $T$-invariant time-like subspaces of $V_t$. 
These subspaces consist of a variation of the orthogonal indecomposable subspaces introduced in Lemma 2.1 of \cite{GK}.
\begin{teo}\label{teo5}
Let $T\in \mathrm{O}(Q)$ be an isometry. Let $W\subset V_t$ be a minimal $T$-invariant time-like  subspace of $V_t$. 
Then one of the following conditions is satisfied.
\begin{enumerate}[1)]
 \item \label{a} $V_t=V_{\pm 1}$, $\dim W=1$ and $W$ is generated by a time-like vector of eigenvalue $\pm 1$.
\item \label{b}$V_t=V_{\pm 1}$, $\dim W\geq 2$ and $\Ker(T|_W-\pm I)$ is generated by a light-like eigenvector.
\item \label{c}$V_t=V_\lambda\oplus V_{{\lambda^{-1}}}$, where $\lambda\neq\pm 1$ is real, $\dim W=2$ 
and $W$ is generated by two light-like eigenvectors corresponding to the eigenvalues $\lambda$ and $\lambda^{-1}$.
\end{enumerate}
\end{teo}
\begin{proof}
We consider the two cases of Corollary $\ref{pr4j}$.

Assume first that $V_t=V_1$, and let $U=\Ker(T|_W-I)$. 
We first show that $U$ has no space-like vectors. 
Let $u\in U$. If $Q(u)>0$ then $Sp\{u\}^{\bot}\cap W$ is a time-like $T$-invariant proper subspace of $W$.
This contradicts the minimality of $W$. Hence $Q(u)\leq 0$.
Assume there exists $u\in U$ such that $Q(u)<0$. Then $Sp\{u\}$ is time-like and $T$-invariant,
and by the minimality of $W$, $Sp\{u\}=W$. Therefore $(\ref{a})$ is satisfied.
Assume that every vector in $U=\Ker(T|_W-I)$ is light-like. 
Then $Q|_{U}=0$ and by Lemma $\ref{condics}$, $\dim U=1$. Since $U$ is light-like,
$U\subsetneq W$, and $\dim W\geq 2$. Therefore $(\ref{b})$ is satisfied.
If $V_t={-1}$ we proceed analogously.

Assume now that $V_t= V_{\lambda}\oplus V_{\lambda^{-1}}$, 
where $\lambda\neq\pm 1$ is a real eigenvalue. Let $U_\lambda=\Ker(T-\lambda)$.
Then $\dim U_\lambda=\dim U_{\lambda^{-1}}=1$.
Indeed, by Lemma $\ref{le2}$, $Q|_{U_\lambda}=Q|_{U_{\lambda^{-1}}}=0$, and by Lemma $\ref{condics}$
their dimension is $\leq 1$. Since both $U_\lambda$ and $U_{{\lambda^{-1}}}$ are not zero, the 
subspace $U_\lambda\oplus U_{\lambda^{-1}}$ has dimension 2, and it is time-like and $T$-invariant.
By the minimality of $W$ we have $W=U_\lambda\oplus U_{\lambda^{-1}}$, and $(\ref{c})$ is satisfied.
\end{proof}

\subsection{Isometries of the hyperbolic space}\label{hypernor}
Let $V=\RR^{n+1}$ and $Q(x,y)=- x_0y_0+\sum_{i=1}^n x_iy_i$. 
Then $Q$ has signature $(1,n)$ and we denote $\mathrm{O}(1,n):=\mathrm{O}(Q)$. 
The pair $(\RR^{n+1},Q)$ is the \textit{standard Lorentz space} and is denoted by $\RR^{1,n}$.
In the canonical basis of $\RR^{n+1}$ the matrix of $Q$ is $J=-\mathbb{I}_1\oplus \mathbb{I}_n$
and $\mathrm{O}(1,n)$ is identified with the group of real square matrices $A$ of size $n+1$ such that
$A^t JA=J.$

The hyperboloid model of the $n$-hyperbolic space is
$$\HH^n=\{x\in \RR^{1,n};Q(x)=-1, x_0>0\}.$$
The metric of $\RR^{1,n}$ induces a Riemannian metric on $\HH^n$.

The isometry group $\mathbf{I}(\HH^n)$ is isomorphic to the \textit{positive Lorentz group} $\mathrm{O}^+(1,n)$.
These are the matrices of $\mathrm{O}(1,n)$ having a positive entry in the upper left corner.

\begin{defi}
A \textit{Lorentzian normal form} is a matrix of one of the following types:
\begin{enumerate}[1)] 
\item \textit{Elliptic}: $1\oplus A
$, where $A\in \mathrm{O}(n)$ is an orthogonal normal form.
\item \textit{Parabolic}: $\Theta\oplus A
$, where $A\in \mathrm{O}(n-2)$ is an orthogonal normal form and $\Theta=\left(\begin{smallmatrix}
3/2&1&\text{-}1/2\\
1&1&\text{-}1\\
1/2&1&1/2
 \end{smallmatrix}\right).$
 \item \textit{Hyperbolic}:
$\Omega_t\oplus A$, where $\Omega_t=\left(
\begin{matrix}
\cosh t&\sinh t\\
\sinh t& \cosh t
\end{matrix}
\right),\,t\neq 0$ and $A\in \mathrm{O}(n-1)$ is an orthogonal normal form.
\end{enumerate}
\end{defi}
Note that besides the normal forms expected by analogy with the orthogonal group,
which can be reduced to the diagonal form in
the field of complex numbers, there is a normal form, corresponding to the parabolic case,
whose first block is a $3\times 3$ matrix which cannot be diagonalized even in the field of complex
numbers.
\begin{teo}\label{canoniqueshy}
Every element of $\mathrm{O}^+(1,n)$ is conjugate to a Lorentzian normal form, unique up to block permutation.
\end{teo}
\begin{proof}
Let $V=V_t\oplus V_s$ be the space-time decomposition of $V$ with respect to $T$. Since
$Q(V_t,V_s)=0$ we can find an orthogonal basis of $V_s$ and $V_t$ separately. Since $T_s$ is orthogonal,
it has an orthogonal normal form. We can assume that $V=V_t$ and $T=T_t$.
We proceed by induction over $n$. For $n=0$ it is trivial. Assume that $n>0$.
Let $W$ be a minimal time-like $T$-invariant subspace of $V$. If $\dim W<\dim V$ then $W^{\bot}$ is space-like,
and the theorem is true for both $T|_W$ and $T|_{W^{\bot}}$. We can assume that $W=V$. We consider the cases of Theorem $\ref{teo5}$
with positive eigenvalues.

Case 1. If $W$ is generated by a time-like eigenvector of eigenvalue 1, then $T$ is the identity.

Case 2. Assume that $\Ker(T-I)$ is generated by a light-like vector, and that $\dim W\geq 2$. Let $N=T-I$. Since $T$ is an isometry we have
$Q(Nx,Ny)=-Q(Nx,y)-Q(x,Ny).$
Let $F_k=\Ker N^k$. Note that $F_1=\Ker(T-I)$ has dimension $1$ and it is generated by a light-like vector.
Therefore the Jordan normal form of $T$ consists of a single Jordan block, and the $T$-invariant subspaces
are of the form $F_k$ for some $k$. Let us show that $W=F_3$.

If $F_1=F_2$ then $F_1=W$. Since $\dim W\geq 2$ we get a contradiction.
Therefore $F_1\subsetneq F_2$ and there exists $v\in F_2\setminus F_1$.
Let $u=N(v)$. Then $0=Q(u)=Q(Nv)=-2Q(v,Nv)=-2Q(v,u).$
Consequently $Q(u,v)=0$ and by Lemma $\ref{condics}$, $Q(v)>0$. 
We can assume that $Q(v)=1$. Then $W\neq F_2$ and $F_2\subsetneq F_3$, 
since if $F_2=F_3$ we would have $F_2=W$, but $F_2$ is not time-like.
Let $w\in F_3\setminus F_2$. We can assume that $v=N(w)$. Then
$$0=Q(u,v)=Q(N^2w,Nw)=-Q(w,N^2w)-Q(Nw,Nw)=-Q(u,w)-1.$$
Therefore $Q(w,u)=-1\neq 0$. By adding to $w$ a multiple of $u$ we can assume that $Q(w)=0$. Then
$Q(v)=Q(Nw)=-2Q(w,Nw)=-2Q(w,v),$
and hence $Q(v,w)=-{1\over 2}$. The subspace $F_3=Sp\{ u,v,w\}$ is time-like and $T$-invariant. 
By the minimality of $W$ we get $F_3=W$.

The matrix of $T$ in the basis $\{u,v,w\}$ is a
Jordan block of size 3 and eigenvalue 1. Let
$$G:=\left(\begin{smallmatrix}
0&0&\text{-}1\\
0&1&\text{-}{1\over2}\\
\text{-}1&\text{-}{1\over2}&0
          \end{smallmatrix}\right)\text{ and }P:=\left(
\begin{smallmatrix}
{3\over 8}&0&{5\over 8}\\
{1\over 2}&1&\text{-}{1\over 2}\\
1&0&\text{-}1
\end{smallmatrix}\right).$$
The matrix of $Q$ in the basis $\{u,v,w\}$
is $G$.
Let $J=-\mathbb{I}_1\oplus \mathbb{I}_2$.
Since $P^tGP=J$, the basis
$\{e_0,e_1,e_2\}$ of $W$ defined by $(e_0\,e_1\,e_2)=(u\,v\,w)P$ 
is a Lorentz orthonormal basis of $W$. The matrix of $T$ in this basis is
$$\Theta:=P^{-1} \left(
\begin{smallmatrix}
1&1&0\\
0&1&1\\
0&0&1
\end{smallmatrix}\right)P=\left(
\begin{smallmatrix}
{3\over 2}&1&\text{-}{1\over 2}\\
{1}&1&\text{-}{1}\\
{1\over 2}&1&{1\over 2}
\end{smallmatrix}\right).
$$
Case 3. Assume $W=Sp\{ u,v\}$, where $u,v$ are light-like eigenvectors of 
eigenvalues $\lambda$ and $\lambda^{-1}$ respectively, where $\lambda\neq 1$ 
is real positive. By Lemma $\ref{condics}$, $Q(u,v)\neq 0$, and we can choose
$u,v$ such that $Q(u,v)=-1$. Then the basis of $W$ given by 
$\{{1\over{\sqrt 2}}(u+v),{1\over{\sqrt 2}}(u-v)\}$ is Lorentz-orthonormal and the matrix of $T$ in this basis is 
$$\Omega_t:=\left(\begin{matrix}
\cosh t&\sinh t\\
\sinh t& \cosh t
                 \end{matrix}\right)
,\text{ where }t=\log \lambda.$$
By considering negative eigenvalues in the above three cases, we obtain the same normal forms with a change of sign,
thus giving non-positive isometries.
\end{proof}

\begin{defi}
Let $f\in \mathbf{I}(\HH^n)$. By Theorem $\ref{canoniqueshy}$,
$f$ has a normal form $T=T_t\oplus T_s$
where $T_s\in \mathrm{O}(n+1-r)$ is an orthogonal normal form and $T_t\in \mathrm{O}^+(1,r-1)$ satisfies one of the following:
\begin{enumerate}[1)]
 \item \label{el1}\textit{Elliptic}: $T_t=\mathbb{I}_r$ with $1\leq r\leq n+1$.
\item \label{pa1}\textit{Parabolic}: $T_t=\Theta\oplus \mathbb{I}_{r-3}$ with $3\leq r\leq n+1$.
\item \label{hy1}\textit{Hyperbolic}: $T_t=\Omega_t$ and $r=2$.
\end{enumerate}
The \textit{Segre symbol of $f$} 
is defined by $\sigma_f=[r;\tilde\sigma;t]$, where $\tilde\sigma$ is the Segre symbol of $T_s$
and $t\in\{e,p,h\}$ denotes the type of isometry (elliptic, parabolic or hyperbolic).
\end{defi}

The following result is obtained in \cite{GK} by other methods, in the context of $z$-classes.
\begin{prop}
The number of Segre classes $h(n)$ of $\mathbf{I}(\HH^n)$ is given by
$h(n)=h_e(n)+h_p(n)+h_h(n)$, where $h_e(n)$, $h_p(n)$ and $h_h(n)$ denote the number of Segre classes of elliptic, 
parabolic and hyperbolic isometries respectively and
$$h_e(n)=\sum_{j=0}^{[n/2]} p(j)(n-2j+1);\,h_p(n)=h_e(n-2);\, h_h(n)=s(n-2).$$
Here $p(k)$ is the partition number of $k$ and $s(k)$ is the number of Segre classes of $\mathbf{I}(\Sp^k)$.
\end{prop}
\begin{proof}
The computation of $h_e(n)$ is analogous to the computation $e_e(n)$ of elliptic isometries of the Euclidean space of Proposition $\ref{segreclasseseuclidean}$.
The remaining identities are trivial.
\end{proof}

\section{$z$-classes of isometries}\label{seccioorbittypes}
Let $G$ be a Lie group acting on itself by conjugation. 
The orbit classes in this action are called \textit{centralizer classes} or \textit{$z$-classes}.
Given $x\in G$ we will denote by $Z(x)=\{y\in G; xy=yx\}$ the \textit{centralizer of $x$ in $G$}.
Then two elements $x,y\in G$ are in the same $z$-class if and only if $Z(x)$ is conjugate to $Z(y)$. 
In this section we study the $z$-classes of the isometry groups by
computing the centralizers of normal forms and their dimensions.
The main result is that
the decomposition of isometries by $z$-classes coincides 
with the Segre decomposition.
In particular, the $z$-classes are parameterized by the Segre symbol.
This allows to compute the dimensions of the strata in terms of the Segre symbol.

\subsection{$z$-classes of $\mathbf{I}(\Sp^n)$}
Let us first recall some well known results
on primary decompositions. This will allow to reduce our computations to normal forms having a single primary component.

\begin{prop}\label{commutadiag}
Let $V$ be a finite dimensional vector space over a field $\kk$ and let $A\in End(V)$ be a linear map.
Let $V=V_1\oplus \cdots\oplus V_t$ be a primary decomposition of $V$ with respect to $A$.
If $B\in End(V)$ is a linear map which commutes with $A$, then $V_i$ is $B$-invariant for $1\leq i\leq t$.
\end{prop}
\begin{proof}
Denote by $\mu_i$ the minimal polynomial of $V_i$ with respect to $A$.
For $x\in V$ we have that
$\mu_i(A)x=0$ if and only if $x\in V_i$. Since $AB=BA$ we have $\mu_i(A)B=B\mu_i(A)$. 
Let $x\in V_i$. Then $\mu_i(A)Bx=B\mu_i(A)x=0$. Hence $Bx\in V_i$, and $V_i$ is $B$-invariant.
\end{proof}

\begin{defi}
We say that a matrix $A\in \mathrm{GL}(n,\kk)$ is \textit{eigenordered} if $A=A_1\oplus\cdots\oplus A_t$ and the
minimal polynomials of $A_i$ are coprime by pairs. We call this an \textit{eigendecomposition} of $A$.
\end{defi}

\begin{cor}\label{eigenordered}
Let $A$ be an eigenordered matrix and let $A=\bigoplus_{i=1}^t A_i\in \mathrm{GL}(n,\kk)$ be 
an eigendecomposition of $A$. Let $B$ be a square $(n\times n)$ matrix such that $[A,B]=0$. Then
$B=\bigoplus_{i=1}^t B_i$, and $[A_i,B_i]=0$ for all $i=1,\cdots,t$.
\end{cor}

We denote by $\mathrm{U}(n)$ the subgroup of $\mathrm{GL}(n,\CC)$ of complex unitary square 
matrices of size $n$. Define an inclusion
$j:\mathrm{GL}(n,\CC)\to \mathrm{GL}(n,2\RR)$ as follows:
given $M=(m_{kl})\in \mathrm{GL}(n,\CC)$ let $m_{kl}=a_{kl}+ib_{kl}$, where $a_{kl},b_{kl}\in\RR$. Then
$j(M)$ is the block matrix $j(M)=\left(\begin{smallmatrix}a_{kl}&\text{-}b_{kl}\\b_{kl}&a_{kl}\end{smallmatrix}\right)$.

We obtain the following characterization of centralizers of orthogonal normal forms.
\begin{prop}\label{stabort}
Let $A\in \mathrm{O}(n)$ be a normal form and let $\tilde\sigma_A=[(n_1\bar{n}_1), \cdots,(n_s\bar{n}_s),m_1,m_2]$ be its Segre symbol.
Then 
$Z(A)=\prod_{i=1}^s j(\mathrm{U}(n_i))\times \mathrm{O}(m_1)\times \mathrm{O}(m_2)$ and
$$\dim Z(A)=\sum_{i=1}^s n_i^2+{{1}\over{2}}\sum_{i=1}^2 m_i(m_i-1).$$
\end{prop}
\begin{proof}
Let $A=\bigoplus_{i=1}^k A_i$ be an eigendecomposition of $A$, where $A_i\in \mathrm{O}(k_i)$. 
By Corollary $\ref{eigenordered}$, the matrices commuting with $A$ are block diagonal
matrices $B=\bigoplus_{i=1}^s B_i$ such that $[A_i,B_i]=0$ for all $i$. 
Therefore $Z(A)=\prod_{i=1}^k Z(A_i)$ and $\dim Z(A)=\sum_{i=1}^k \dim Z(A_i)$.
We can reduce to the case when $A$ has a single primary component. 
If $A=\pm\mathbb{I}_n$ the result is trivial. 
Let $n=2k$ and assume that $A=R^{\oplus k}_\theta$. Computing the matrices that commute with $A$ we find that
$Z(A)=j(\mathrm{U}(k))$. Hence $\dim Z(A)=\dim \mathrm{U}(k)=k^2.$
\end{proof}

We next use the above result to prove that the decomposition by $z$-classes coincides with the Segre decomposition.
We first study the case of an orthogonal normal form having a single primary component.

\begin{lem}\label{unvaport}
Let $A\in \mathrm{O}(n)$ be a normal form with a single primary component, 
and let $B\in \mathrm{O}(n)$ such that $Z(A)=Z(B)$. Then $B$ is a normal form with one
primary component and $\tilde\sigma_A=\tilde\sigma_B$.
\end{lem}
\begin{proof}
If $A=\pm\mathbb{I}_n$ then $Z(A)=\mathrm{O}(n)$. Since the center of $\mathrm{O}(n)$ is $\pm\mathbb{I}_n$ it follows that
$B=\pm \mathbb{I}_n$. Therefore $\tilde\sigma_A=\tilde\sigma_B=[n]$.

Let $n=2k$ and assume that $A=R^{\oplus k}_\theta$. By
Proposition $\ref{stabort}$ we have $Z(A)=j(\mathrm{U}(k))$. Since $B\in Z(A)$, 
there exists $C\in \mathrm{U}(k)$ such that $B=j(C)$. Since $Z(B)=j(\mathrm{U}(k))$ the centralizer of $C$ in $\mathrm{U}(k)$ is the whole group $\mathrm{U}(k)$. 
Therefore $C=e^{i\alpha}\mathbb{I}_n$ for some $\alpha\in \RR$. Since $B\neq \pm \mathbb{I}_{n}$
we have $\alpha\in\RR\setminus\{k\pi\}$ and $B=R^{\oplus k}_\alpha$. In particular, $\tilde\sigma_A=\tilde\sigma_B=[(k\bar k)]$.
\end{proof}
\begin{teo}\label{otort}
Let $f,g\in \mathbf{I}(\Sp^n)$. Then $\tilde\sigma_f=\tilde\sigma_g$ if and only if $f$ and $g$ are in the same $z$-class.
\end{teo}
\begin{proof}
Assume that $\tilde\sigma_f=\tilde\sigma_g$. By Proposition $\ref{stabort}$ we can find normal forms $A$ and $B$ of $f$ and $g$ respectively 
such that $Z(A)=Z(B)$. Since conjugate matrices have conjugate centralizers it follows that $Z(f)$ and $Z(g)$ are conjugate.

Conversely, assume that $Z(f)$ and $Z(g)$ are conjugate. 
Let $A$ be a normal form of $f$ such that $Z(A)=Z(B)$, where $B$ is the matrix of $g$ in some basis.
Take an eigendecomposition $A=\bigoplus_{i=1}^k A_i$. By Proposition $\ref{eigenordered}$, $B$ admits a block diagonal decomposition
$B=\bigoplus_{i=1}^k B_i$
such that $Z(A_i)=Z(B_i)$. 
By Lemma $\ref{unvaport}$ the matrices $B_i$ are normal forms and $\tilde\sigma_{A_i}=\tilde\sigma_{B_i}$. 
Therefore if $A=R_{\theta_1}^{n_1}\oplus\cdots\oplus R_{\theta_s}^{n_s}\oplus \pm(\mathbb{I}_{m_1}\oplus -\mathbb{I}_{m_2})$ 
then $B=R_{\alpha_1}^{n_1}\oplus\cdots\oplus R_{\alpha_s}^{n_s}\oplus \mu\mathbb{I}_{m_1}\oplus \nu\mathbb{I}_{m_2}.$
In order to show that $\tilde\sigma_A=\tilde\sigma_B$ we need to see that $\alpha_i\neq \alpha_j$ for all $i\neq j$, and that $\mu\neq \nu$. 
This follows from the description of the centralizers of normal forms of Proposition $\ref{stabort}$, since
$j(\mathrm{U}(n_i))\times j(\mathrm{U}(n_j))\varsubsetneq j(\mathrm{U}(n_i+n_j))$ and $\mathrm{O}(m_1)\times \mathrm{O}(m_2)\varsubsetneq \mathrm{O}(m_1+m_2).$
\end{proof}

Since the group of isometries $\mathbf{I}(\Sp^n)$ is compact, its decomposition by $z$-classes is a Whitney-regular stratification
(see for example \cite{Pf}). In particular, each stratum is a manifold. We give their dimensions.

\begin{prop}\label{dimsigmaspheres}
Let $\Sigma$ be a Segre stratum of $\mathbf{I}(\Sp^n)$ with ${\tilde\sigma}=[(n_1\bar{n}_1), \cdots,(n_s\bar{n}_s),m_1,m_2]$.
Then
$$\dim\Sigma ={1\over 2}n(n+1)-\sum_{i=1}^s n_i^2-{1\over 2}\sum_{i=1}^2m_i(m_i-1)+s.$$
\end{prop}
\begin{proof}
Let $f$ be an element of $\Sigma$ and $[f]$ its orbit. 
The number of parameters of a normal from of $f$ is $s$. Therefore
$\dim\Sigma=\dim [f]+s$. Since $\dim [f]=\dim \mathrm{O}(n+1)-Z(f)$,
the result follows from Proposition $\ref{stabort}$.
\end{proof}
Each connected component of $\mathbf{I}(\Sp^n)$ has a stratum of maximal dimension $\Sigma_*$ which is dense in $\mathbf{I}(\Sp^n)$.
In view of the above formula, the maximal strata are those with Segre symbol
${\tilde\sigma}_*=[(1\bar{1}), \cdots,(1\bar{1})]$ and ${\tilde\sigma}_*=[(1\bar{1}), \cdots,(1\bar{1}),1,1]$ if $n$ is even and 
${\tilde\sigma}_*=[(1\bar{1}), \cdots,(1\bar{1}),1]$ if $n$ is odd.

\subsection{$z$-classes of $\mathbf{I}(\EE^n)$}
Consider the decomposition of a normal form $A\in\Euc(n)$ into its orthogonal and unipotent parts
$A=A_R\oplus A_1$, 
where $A_R\in \mathrm{O}(k)$ 
is an orthogonal normal form with no eigenvalues equal to 1 and $A_1\in\Euc(n-k)$ is a unipotent Euclidean normal form. 
Then the centralizer of $A$ in $\Euc(n)$ decomposes as $Z(A)=Z(A_R)\times Z(A_1)$,
where $Z(A_R)\in \mathrm{O}(k)$ and $Z(A_1)\in \Euc(n-k)$ denote the centralizers of $A_R$ and $A_1$ in 
$\mathrm{O}(k)$ and $\Euc(n-k)$ respectively.
Therefore to study the centralizer of $A$ it suffices to consider the case for which $A=A_1$ is unipotent.

\begin{prop}\label{stabeuc}
Let $A\in\Euc(n)$ be a unipotent Euclidean normal form.
\begin{enumerate}[1)]
\item If $A=\mathbb{I}_{n+1}$ is elliptic then $Z(A)=\Euc(n)$ and $\dim Z(A)={1\over 2}n(n+1)$.
\item If $A=\mathbb{I}_{n-1}\oplus \left(\begin{smallmatrix}
1&a\\
0&1
\end{smallmatrix}\right)$ is hyperbolic then 
$$Z(A)=\left\{\left(
\begin{smallmatrix}
B&0&b\\
0&1&c\\
0&0&1
\end{smallmatrix}
\right)\,;\,B\in \mathrm{O}(n-1), \begin{array}{l}b\in \RR^{n-1}\\ c\in \RR\end{array}
\right\}\text{ and }\dim Z(A)={1\over 2}n(n-1)+1.$$
\end{enumerate}
\end{prop}
\begin{proof}
It follows by a simple computation of the matrices commuting with $A$ in each case.
\end{proof}
\begin{lem}\label{arregloeuc}
Let $A\in \Euc(n)$ be a unipotent normal form and let $B\in \Euc(n)$ such that $Z(A)=Z(B)$. Then $B$ is unipotent and $\sigma_A=\sigma_B$.
\end{lem}
\begin{proof}
If $A=\mathbb{I}_{n+1}$ it is trivial. 
Assume that $A=\mathbb{I}_{n-1}\oplus\left(
\begin{smallmatrix}
1&a\\
0&1
\end{smallmatrix}\right)$. 
If $n=1$ then $B=\left(\begin{smallmatrix}1&b\\0&1\end{smallmatrix}\right)$ for some $b\in\RR$. Since $B$ cannot be the identity,
we have $b\neq 0$ and $\sigma_A=\sigma_B$.

Assume that $n>1$, and let $C=C_1 \oplus\left(\begin{smallmatrix}1&b\\0&1\end{smallmatrix}\right)\in\Euc(n)$, where
$C_1\in \mathrm{O}(n-1)$ and $b\in\RR$ are arbitrary. By Proposition $\ref{stabeuc}$, $C\in Z(A)$ and by hypothesis $C\in Z(B)$. Hence $[C,B]=0$. 
Therefore $B$ is a block diagonal matrix $B=B_R\oplus B_1$, where $B_R\in \mathrm{O}(n-1)$ and $B_1\in \Euc(1)$ 
are such that $Z(B_R)=\mathrm{O}(n)$ and $[B_1,\left(\begin{smallmatrix}1&b\\0&1\end{smallmatrix}\right)]=0$.
Hence $B_R=\mathbb{I}_{n-1}$ and by the case $n=1$, $B_1=\left(\begin{smallmatrix}1&c\\0&1\end{smallmatrix}\right)$,
for some $c\in \RR$. Since $B\neq \mathbb{I}_2$ we have $c\neq 0$. Hence $\sigma_A=\sigma_B$.
\end{proof}
\begin{teo}\label{zclasseuc}
Let $f,g\in \mathbf{I}(\EE^n)$. Then $\sigma_f=\sigma_g$ if and only if $f$ and $g$ are in the same $z$-class.
\end{teo}
\begin{proof}
Assume that $\sigma_f=\sigma_g$. By Proposition $\ref{stabeuc}$ there exist normal 
forms $A$ and $B$ of $f$ and $g$ respectively, such that $Z(A)=Z(B)$.
Since conjugate matrices have conjugate isotropy groups it follows that
$Z(f)$ is conjugate to $Z(g)$.

Conversely, assume that $Z(f)$ is conjugate to $Z(g)$. There exists a normal form $A\in\Euc(n)$ of $f$ and
a matrix $B$ of $g$ such that $Z(A)=Z(B)$. Let $A=A_R\oplus A_1$ be the decomposition of $A$ into its
orthogonal and unipotent components. Then $B$ is a block diagonal matrix $B=B_R\oplus B_1$, 
where $B_R\in \mathrm{O}(k)$ and $B_1\in\Euc(n-k)$ are such that $Z(B_R)=Z(A_R)$ and $Z(B_1)=Z(A_1)$. 
By Theorem $\ref{otort}$ we have $\tilde\sigma_{A_R}=\tilde\sigma_{B_R}$. 
By Lemma $\ref{arregloeuc}$ we have $\sigma_{A_1}=\sigma_{B_1}$. In particular $B_1$ is unipotent.
To prove that $\sigma_A=\sigma_B$ it remains to show that $1$ is not an eigenvalue of $B_R$.

Assume that $A_R=R\oplus\text{-}\mathbb{I}_m$, where $R\in \mathrm{O}(k-m)$ is a rotation matrix
of the form $R=R_{\theta_1}^{n_1}\oplus\cdots\oplus R_{\theta_s}^{n_s}$. 
Since $\tilde\sigma_{A_R}=\tilde\sigma_{B_R}$, $B_R$ has the form $R'\oplus \eps\mathbb{I}_m$, where $\eps=\pm 1$ and $\tilde\sigma_R=\tilde\sigma_{R'}$.
Assume that $\eps=1$. Then
$d_B:=\dim Z(B)=\dim Z(R)+\dim Z({\mathbb{I}_m\oplus A_1})$ and $d_A:=\dim Z(A)=\dim Z(R)+\dim \mathrm{O}(m)+\dim Z(A_1)$.
By Proposition $\ref{stabeuc}$ we have $d_B-d_A>0$, which is a contradiction.
Therefore $\eps=-1$ and $\sigma_B=\sigma_A$. Since $A$ and $B$ are matrices of $f$ and $g$, we have $\sigma_f=\sigma_g$.
\end{proof}

\begin{prop}\label{dimsigmaeuclidi}
Let $\Sigma$ be the $z$-stratum of $\mathbf{I}(\EE^n)$ with Segre symbol $\sigma=[\tilde\sigma;r;t]$. Denote by
$\widetilde\Sigma$ the stratum of $\mathbf{I}(\Sp^{n-r-1})$ with Segre symbol $\tilde\sigma$.
\begin{enumerate}[1)]
\item If $\Sigma$ has elliptic type then $\dim\Sigma=\dim\widetilde\Sigma+(n-r)(r+1)$.
\item If $\Sigma$ has hyperbolic type $\dim\Sigma=\dim\widetilde\Sigma+n(r+1)-r^2$.
\end{enumerate}
\end{prop}
\begin{proof}
Let $f$ be an element of $\Sigma$ and $[f]$ its orbit. Then 
$\dim\Sigma=\dim [f]+\Delta(f)$, where $\Delta(f)$ is the number of
parameters of a normal form of $f$. 
If $f$ is elliptic then $\Delta(f)=\Delta(\tilde f)$,
while if $f$ is hyperbolic
 $\Delta(f)=\Delta(\tilde f)+1$, since the translation involves one parameter.
Here $\tilde f$ denotes an element of $\widetilde\Sigma$. 
Since $\dim [f]=\dim \Euc(n)-Z(f)$, the result follows from Propositions $\ref{dimsigmaspheres}$ and $\ref{stabeuc}$.
\end{proof}

\subsection{$z$-classes of $\mathbf{I}(\HH^n)$}
Consider the space-time decomposition $T=T_t\oplus T_s\in \mathrm{O}^+(1,n)$ of a Lorentzian normal form.
Let $Z(T_t)$ and $Z(T_s)$ denote the centralizers of $T_t$ and $T_s$ in $\mathrm{O}^+(1,n_t)$ and $\mathrm{O}(n_s)$ respectively.
Then by Corollary $\ref{eigenordered}$ we have $Z(T)=Z(T_t)\times Z(T_s)\in\mathrm{O}^+(1,n)$.
The centralizers of spatial components were studied previously. We next study the temporal component.
\begin{prop}\label{stalor}
Let $T\in \mathrm{O}^+(1,n)$ be a Lorentzian normal form with no spatial component. 
\begin{enumerate}[1)]
 \item If $T=\mathbb{I}_{n+1}$ is elliptic then $Z(T)=\mathrm{O}^+(1,n)$ and $\dim Z(T)=n(n+1)/2$.
\item  If $T=\Theta\oplus \mathbb{I}_{n-2}$ is parabolic then
$$Z(T)=\left\{
{\tiny \left(\begin{array}{@{\,\,}c@{\,\,}c@{\,\,}c@{\,\,}c@{\,\,}c@{\,\,}c}
1\text{\normalfont{+}}d&c&\text{-}d&\vline&\textbf{a}\\
c&1&\text{-}c&\vline&0\\
d&c&1\text{-}d&\vline&a\\
\hline
{b}&0&\text{-}b&\vline&D
\end{array}\right);}
\begin{array}{l}
D\in \mathrm{O}(n-2)\\
{a}\in\RR^{n-2}\\
c\in\RR
\end{array},
\begin{array}{l}
d={1\over 2}(c^2+||{b}||^2)\\
{a}={b}D
\end{array}
\right\}
$$ and $\dim Z(T)=1+{1\over 2}(n-2)(n-1)$.
\item  If $T=\Omega_t$ is hyperbolic then $Z(T)$ is the component of $\mathbb{I}_2$ in $\mathrm{O}^+(1,1)$ and $\dim Z(T)=1$.
\end{enumerate}
\end{prop}
\begin{proof}
The only non-trivial case is the second. Let $M\in Z(T)$. 
We write $M$ as a block matrix $M=\left(\begin{smallmatrix}C&A\\B&D\end{smallmatrix}\right)$. The condition $[M,T]=0$ implies
$[\Theta,C]=0$, $(\Theta-\mathbb{I}_{3})A=0$ and $ B(\Theta-\mathbb{I}_3)=0.$
An easy computation shows that
$$A=\left(\begin{smallmatrix}
a_1&\cdots&a_{n-2}\\
0&\cdots&0\\
a_1&\cdots&a_{n-2}
          \end{smallmatrix}
\right)\text{ and }
B=\left(\begin{smallmatrix}
b_1&0&\text{-}b_1\\
\vdots&\vdots&\vdots\\
b_{n-2}&0&\text{-}b_{n-2}
          \end{smallmatrix}
\right).
$$
Recall from Section $\ref{hypernor}$ that $\Theta$ is conjugate to a unipotent Jordan block of size 3. 
The matrices commuting with a Jordan block are regular upper triangular matrices (see \cite{Ga}). This implies that
$C$ is a matrix of the form
$$C=\left(\begin{smallmatrix}
e\text{+}d&c&\text{-}d\\
c&e&\text{-}c\\
d&c&e\text{-}d
          \end{smallmatrix}
\right).$$
The condition $M^tQM=Q$ implies:
$C^tQC+B^tB=Q$, $A^tQA+D^tD=\mathbb{I}_{n-2}$ and $C^tQA+B^tD=0$.
From the first of these identities we deduce that $e=\pm 1$, and that $d={1\over 2e}(c^2+\sum_{i=1}^{n-2} b_i^2)$. 
Furthermore, for $M$ to be positive, the upper left entry of $C$ must be positive. Hence $e=1$.
Since $A^tQA=0$, from the second identity we obtain $D^tD=\mathbb{I}_{n-2}$, therefore $D\in \mathrm{O}(n-2)$. 
The third identity implies that $a_j=\sum_{i=1}^{n-2} b_id_{ij}$, where $d_{ij}$ are the coefficients
of $D$.
\end{proof}
\begin{lem}\label{arreglopoi}
Let $T\in \mathrm{O}^+(1,n)$ be a Lorentzian normal form with no spatial component and let $A\in \mathrm{O}^+(1,n)$ such that $Z(T)=Z(A)$. Then $\sigma_T=\sigma_A$.
\end{lem}
\begin{proof}
The only non-trivial case is $T=\Theta\oplus\mathbb{I}_{n-2}$. Let $B=\mathbb{I}_3\oplus C$, where $C\in \mathrm{O}(n-2)$. 
Then $B\in Z(T)$ and $[B,A]=0$. By Corollary $\ref{eigenordered}$, $A$ 
is a block diagonal matrix $A=A_1\oplus A_2$, where $[A_1,\Theta]=0$ and $[A_2,C]=0$ for all $C\in \mathrm{O}(n-2)$. Therefore $A_2=\pm \mathbb{I}_{n-2}$ and 
$$A_1=\left(\begin{smallmatrix}
a\text{+}d&b&\text{-}d\\
b&a&\text{-}b\\
d&b&a\text{-}d
        \end{smallmatrix}\right).$$
The condition $A\in \mathrm{O}^+(1,n)$ implies $a=1$ and $d=b^2/2$.
If $b=0$ then $A_1$ is the identity matrix and by Proposition $\ref{stalor}$, $Z(A)\neq Z(T)$.
Therefore $b\neq 0$. In such case, $A_1$ is conjugate to $\Theta$, and $A$ has a normal form $\Theta\oplus\eps\mathbb{I}_{n-2}$, where $\eps=\pm 1$.
To see that $\sigma_A=\sigma_T$ it remains to show that for all $n>2$ we have $\eps=1$.

Assume that $n>2$ and $\eps=-1$. Then by Proposition $\ref{stalor}$, $Z(A)\cong Z(\Theta)\times \mathrm{O}(n-2)$. In particular we have
$\dim Z(T)-\dim Z(A)=n-2>0$, which is a contradiction. Hence $\eps=1$ and $\sigma_A=\sigma_T$. 
\end{proof}
\begin{teo}\label{zclasshyp}
Let $f,g\in \mathbf{I}(\HH^n)$. Then $\sigma_f=\sigma_g$ if and only if $f$ and $g$ are in the same $z$-class.
\end{teo}
\begin{proof}
Assume that $\sigma_f=\sigma_g$. By Proposition $\ref{stalor}$,
there exist normal forms $A$ and $B$ of $f$ and $g$ respectively 
such that $Z(A)=Z(B)$. Therefore the isotropy groups of $f$ and $g$ are conjugate to each other.

Conversely, assume that $Z(f)$ is conjugate to $Z(g)$. Then there exists a normal form $T$ of
$f$ such that $Z(T)=Z(B)$, where $B$ is the matrix of $g$ in some Lorentz basis.
Let $T=T_t\oplus T_s$ be the space-time decomposition of $T$. Then $B$ is a block diagonal matrix $B=B_1\oplus B_s$, where
$Z(T_t)=Z(B_1)$ and $Z(T_s)=Z(B_2)$. By Theorem $\ref{otort}$ we have $\tilde\sigma_{T_s}=\tilde\sigma_{B_2}$. By Lemma $\ref{arreglopoi}$ 
we have $\sigma_{T_t}=\sigma_{B_1}$.
To prove that $\sigma_B=\sigma_T$ it remains to show that if $T$ is elliptic or parabolic, then $1$ is not an eigenvalue of $B_2$.

Let $T_s=R\oplus\text{-}\mathbb{I}_{k}$, where $R$ is a composition of rotations. 
Since $\tilde\sigma_{T_s}=\tilde\sigma_{B_2}$, $B_2$ has a normal form
$R'\oplus \eps\mathbb{I}_k$, where $\tilde\sigma_R=\tilde\sigma_{R'}$. 
Assume that $\eps=1$. Then
$$\begin{array}{l}
d_B=\dim Z(B)=\dim Z({T_t\oplus\mathbb{I}_k})+\dim Z(R),\\
d_T=\dim Z({T})=\dim Z({T_t)}+\dim Z(R)+\dim \mathrm{O}(k).\end{array}
$$
By Proposition $\ref{stalor}$, in the elliptic and parabolic cases, we find that
$d_B-d_T> 0$, which is a contradiction. Therefore $\eps=-1$, and 1 is not
an eigenvalue of $B_2$.
Since $T$ and $A$ are matrices of $f$ and $g$ respectively, we have $\sigma_f=\sigma_g$.
\end{proof}

\begin{prop}
Let $\Sigma$ be a $z$-stratum of $\mathbf{I}(\HH^n)$ with Segre symbol $\sigma=[r;\tilde\sigma;t]$.
Denote by
$\widetilde\Sigma$ the stratum of $\mathbf{I}(\Sp^{n-r-1})$ with Segre symbol $\tilde\sigma$.

\begin{enumerate}[1)]
\item If $\Sigma$ has elliptic type then $\dim\Sigma=\dim\widetilde\Sigma+nr-r^2+r$.
\item If $\Sigma$ has parabolic type then $\dim\Sigma=\dim\widetilde\Sigma+nr-r^2+3r-4$.
\item If $\Sigma$ has hyperbolic type $\dim\Sigma=\dim\widetilde\Sigma+2n-1$.
\end{enumerate}
\end{prop}
\begin{proof}
Let $f$ be an element of $\Sigma$ and $[f]$ its orbit. Then 
$\dim\Sigma=\dim [f]+\Delta(f)$, where $\Delta(f)$ is the number of
parameters of a normal form of $f$. 
If $f$ is elliptic or parabolic then $\Delta(f)=\Delta(\tilde f)$,
while if $f$ is hyperbolic
then $\Delta(f)=\Delta(\tilde f)+1$, since the hyperbolic rotations depend on a parameter.
Here $\tilde f$ denotes an element of $\widetilde\Sigma$. 
Since $\dim [f]=\dim \Euc(n)-Z(f)$, the result follows from Propositions $\ref{stabort}$ and $\ref{stalor}$.
\end{proof}

\section{Invariant subspaces}
Let $\mathbb{M}^n$ denote one of the Riemannian manifolds $\Sp^n$, $\EE^n$ or $\HH^n$.
Recall that the \textit{generalized Grassmannian} $\Gr{k}{\mathbb{M}^n}$ is the set of closed totally geodesic submanifolds of $\mathbb{M}^n$ 
isometric to $\mathbb{M}^k$. It is a homogeneous space of dimension $(k+1)(n-k)$ (see \cite{Ob}).

Given $f\in \mathbf{I}(\mathbb{M}^n)$  we will denote by $\Gamma_f(k)$ the set of $f$-invariant closed totally geodesic 
submanifolds of $\mathbb{M}^n$ isometric to $\mathbb{M}^k$. Then $\Gamma_f(k)$ is a closed subset of $\Gr{k}{\mathbb{M}^n}$.
We also let $\Gamma_f=\bigsqcup_{k=0}^n \Gamma_f(k)$.

In this section, we describe the sets $\Gamma(k)$ and relate them to the Segre symbol:
We will show that $\Gamma_f(k)$ is a closed smooth submanifold of $\Gr{k}{\mathbb{M}^n}$ and that
given isometries $f,g\in\mathbf{I}(\mathbb{M}^n)$, then $\Gamma_f(k)\cong \Gamma_g(k)$
for all $0\leq k\leq n$ if and only if $f$ and $g$ are in the same Segre class.
In fact, we prove a stronger result: that the dimensions of 
the connected components of $\Gamma_f(k)$, for $0\leq k\leq 4$, determine the Segre symbol (and thus the $z$-class) of $f$.

\subsection{Invariant subspaces of a linear endomorphism}
We begin by recalling some basic properties of invariant subspaces for a linear endomorphism of a 
finite dimensional vector space, which will be needed in the later parts of the section. 
These have been studied by Shayman in \cite{Sh}. In what follows we adopt his notation.

Given an $n$-dimensional vector space $V$ defined over $\RR$ or $\CC$, denote by
$\Grl{k}{V}$ the Grassmannian of $k$-dimensional linear subspaces of $V$, where $k\leq n$.

Given an endomorphism $A\in End(V)$, denote by $S_A(k)$ the set of all $A$-invariant subspaces of $V$ of dimension $k$,
which is a closed algebraic subvariety of $\Grl{k}{V}$.

\begin{lem}\label{similars}
If $A,B\in End(V)$ are conjugate endomorphisms, there is an isomorphism of algebraic varieties $S_A(k)\cong S_B(k)$.
\end{lem}
\begin{proof}
Let $B=\alpha A\alpha^{-1}$ for some $\alpha\in \mathrm{GL}(n)$, 
and let $U\in S_A(k)$. Then $\alpha U\in S_B(k)$. 
The map $U\mapsto \alpha U$ defines an isomorphism $S_A(k)\to S_B(k)$.
\end{proof}
\begin{lem}\label{lemapenjat}
Let $A\in End(V)$ and let $V=V_1\oplus\cdots\oplus V_s$ be 
a primary decomposition of $V$ with respect to $A$.
Then every $A$-invariant subspace $U\subset V$ admits
a primary decomposition $U=U_1\oplus\cdots\oplus U_s$ 
with respect to $A|_U$, such that $U_i=U\cap V_i$.
\end{lem}
\begin{proof}
Let $V_i=\Ker\, p_i(A)^{n_i}$, where $p_i(t)^{n_i}$ is an elementary divisor. 
Then $U_i=\Ker\, p_i(A)^{m_i}$,
with $m_i\leq n_i$. Therefore
$V_i\cap U=\Ker\, p_i(A|_U)^{n_i}=\Ker\, p_i(A|_U)^{m_i}=U_i$.
\end{proof}
The following is a consequence of the previous lemma.
\begin{prop}[\cite{Sh}, Theorem 2]\label{decoautos}
Let $A\in End(V)$ and $V=V_1\oplus\cdots\oplus V_s$ be a primary decomposition of $V$ with respect to $A$.
For all $i=1,\cdots,s$ denote by $\pi_i:V\to V_i$ the natural projection, and by $A_i$ the restriction of $A$ to $V_i$. The map
$S_A\to S_{A_1}\times\cdots\times S_{A_s}$ defined by $U\mapsto(\pi_1(U),\cdots,\pi_s(U))$
is an isomorphism of algebraic varieties. In particular
$$S_A(k)\cong\bigsqcup_{k_1+\cdots +k_s=k} S_{A_1}(k_1)\times\cdots \times S_{A_s}(k_s).$$
\end{prop}
In view of the above result, to study the varieties of invariant subspaces of a linear endomorphism
one can restrict to the case when the endomorphism has a single primary component.

We remark that the varieties $S_A(k)$ are singular in general (see \cite{Sh}).
For orthogonal maps the situation is much more simple, 
since the normal forms of such transformations are diagonal over $\CC$. 
We will see that in this case the varieties $S_A(k)$ are smooth.

\subsection{Invariant totally geodesic subspheres of $\Sp^n$}
The spherical Grassmannian $\Gs{k}{n}$ is the set of $k$-dimensional totally geodesic subspheres
of $\Sp^n$. Such subspheres are precisely the intersections of $\Sp^n$ with the linear subspaces of $\RR^{n+1}$. 
This follows directly from the description of the geodesics in $\Sp^n$. There is an identification
$\Gs{k}{n}=\Grl{k+1}{\RR^{n+1}}.$
Therefore $\Gs{k}{n}$ is a smooth projective variety of dimension $(k+1)(n-k)$.

Let $f\in \mathbf{I}(\Sp^n)$. 
By the above identification we have that $\Sp^k\in \Gas{f}(k)$ if and only if there
exists $U\in S_f(k)$ such that $\Sp^k=\Sp^n\cap U$. This gives an isomorphism
$\Gas{f}(k)\cong S_f(k+1)$.
Therefore the study of $\Gas{f}(k)$ reduces to the study
of the varieties of invariant linear subspaces of a real orthogonal linear map.

\begin{prop}\label{unsolet}Let $A\in \mathrm{O}(n)$ be an orthogonal map having a single primary component.
\begin{enumerate}[i)]
\item If $A=\pm\mathbb{I}_n$ then $S_A(k)\cong \Grl{k}{\RR^n}$.
\item If $A=R^{\oplus p}_\theta$, with $2p=n$ then
$S_A(k)=\emptyset$ if $k$ is odd and $S_A(k)\cong\Grl{{{k}\over{2}}}{\CC^{p}}$ if $k$ is even.
\end{enumerate}
\end{prop}
\begin{proof}
The first case is trivial. Assume that $A=R_\theta^{\oplus p}$.
Let $V=\RR^{n}$ and denote by $V^{c}$ the complexification of $V$. Let $A^{c}:V^{c}\to V^{c}$ denote the complexification of $A$. 
Then $A^{c}$ has a pair of conjugate eigenvalues $\lambda$ and $\overline{\lambda}$. There is a primary
decomposition $V^{c}=V^c_\lambda\oplus \overline V^c_\lambda$, where $V^c_\lambda$ is the eigenspace of 
eigenvalue $\lambda$ and $\overline V^c_\lambda=V_{\overline\lambda}^c$ is the eigenspace of eigenvalue
$\overline \lambda$. Let $A_\lambda=A^{c}|_{V^c_\lambda}$ denote the restriction of $A^{c}$ to $V^c_\lambda$, 
and let $\pi_\lambda:V^{c}\to V^c_\lambda$ be the projection of $V^{c}$ along $\overline V^c_\lambda$. By Lemma $\ref{lemapenjat}$,
there is an isomorphism $S_{{A}}\longrightarrow S_{A_\lambda}$ of real algebraic varieties given by
$W\mapsto \pi_\lambda(W^{c})$,
where $W^{c}$ denotes the complexification of the subspace $W$. Note that the normal form of $A_\lambda$ 
is $\lambda \mathbb{I}_p$, where $\lambda\in \CC$. Therefore $S_{A_\lambda}(k)\cong \Grl{k}{\CC^p}$.
By the above isomorphism we get $S_A(k)\cong S_{A_\lambda}({{k}\over{2}})\cong \Grl{{{k}\over{2}}}{\CC^p}$.
\end{proof}

\begin{prop}\label{lineaconnex}
Let $A\in \mathrm{O}(n)$ be an orthogonal map.
\begin{enumerate}[i)]
 \item If $\tilde\sigma_A=[(n_1\overline n_1),\cdots,(n_s\overline n_s),m_1,m_2]$ is the Segre symbol of $A$ then
$$S_A(k)\cong \bigsqcup_{
\begin{array}{c}
\sum_{i=1}^s 2k_i+r_1+r_2=k\\
 k_i\leq n_i,\, r_i\leq m_i
\end{array}
} \Grl{k_1}{\CC^{n_1}}\times\cdots\times \Grl{k_s}{\CC^{n_s}}\times \Grl{r_1}{\RR^{m_1}}\times \Grl{r_2}{\RR^{m_2}}.$$
\item The dimensions of the connected components of $S_A(1)$ and $S_A(2)$ determine $\tilde\sigma_A$.
\end{enumerate}
\end{prop}
\begin{proof}
The first statement is a consequence of Propositions $\ref{decoautos}$ and $\ref{unsolet}$. Let us prove the second.
The variety of invariant lines of $A$ is given by
$$S_A(1)\cong \Grl{1}{\RR^{m_1}}\sqcup \Grl{1}{\RR^{m_2}}\cong\PP_{\RR}^{m_1-1}\sqcup \PP_{\RR}^{m_2-1}$$ for certain $m_1\geq m_2\geq 0$,
with the convention that $\Grl{1}{\RR^{0}}=\PP_{\RR}^{-1}=\emptyset$.
Likewise, the variety of invariant planes of $A$ is given by
$$S_A(2)\cong \left(\bigsqcup_{i=1}^s \PP_\CC^{n_i-1}\right)\sqcup \Grl{2}{\RR^{m_1}}\sqcup \Grl{2}{\RR^{m_2}}\sqcup  \left(\PP_{\RR}^{m_1-1}\times\PP_{\RR}^{m_2-1}\right)$$
where $s\geq 0$, $n_1\geq \cdots\geq n_s\geq 0$, and again $\Grl{1}{\CC^{0}}=\PP_\CC^{-1}=\emptyset$ and $\Grl{2}{\RR^{m}}=\emptyset$ for $m<2$.

Let $d=(d_1,d_2)$ denote the vector formed by the dimensions of the connected components of $S_A(1)$, where we set $d_1\geq d_2$ and $\dim(\emptyset)=-1$.
We then let $m_1:=d_1+1$ and $m_2:=d_2+1$.
Let $E$ denote the set of dimensions of the connected components of $S_A(2)$. By the above formula,
this set contains the elements $e_1:=m_1(2-m_1)$, $e_2:=m_2(2-m_2)$ (corresponding to the dimensions of the real Grassmannians of planes) and $e_3:=m_1+m_2-2$
(corresponding to the product of two real projective spaces) whenever they are non-negative.
Let $E'=E-E\cap \{e_1,e_2,e_3\}$ and $s:=\#(E')$. We can write the elements of $E'$ as a vector $e=(e_1,\cdots,e_s)$ where $e_1\geq e_2\geq\cdots\geq e_s>0$. 
Note that the elements of $E'$ must be even by construction (since they correspond to real dimensions of complex projective spaces).
For $1\leq i\leq s$, let $n_i:=e_i/2+1$. Then the Segre symbol of $A$ is
$\tilde\sigma_A=[(n_1,\overline n_1),\cdots,(n_s,\overline n_s),m_1,m_2]$, where entries that are 0 are omitted from the notation.
\end{proof}

\begin{teo}\label{main1}
Let $f\in \mathbf{I}(\Sp^n)$. The dimensions of the connected components of $\Gas{f}(0)$ and $\Gas{f}(1)$ determine the Segre symbol of $f$.
\end{teo}
\begin{proof}
It follows from Proposition $\ref{lineaconnex}$ and the fact that $\Gas{f}(k)\cong S_f(k+1)$.
\end{proof}

Tables $\ref{tts1},\ref{tts2}$ and $\ref{tts3}$ show a normal form representative of each Segre class,
of isometries of $\Sp^1$, $\Sp^2$ and $\Sp^3$, together with the varieties of invariant subspheres of each dimension.

Note that the elements of $\Gamma_f(0)$ are $0$-dimensional subspheres of $\Sp^n$, that is, pairs of antipodal points
of $\Sp^n$ which are $f$-invariant, but each point of the pair need not be fixed. One could also consider
a finer classification taking into account the sets of fixed points. Then every $z$-class of isometries
having a real eigenvalue $\pm 1$ would split into two subclasses corresponding to the 
connected components of the strata.
Although $\Sp^1$ is not a
space of constant curvature we include the table of $\mathbf{I}(\Sp^1)$ for completeness.

\begin{center}
\begin{table}[h]
\caption{Isometries of $\Sp^1$
}\label{tts1}
$
\begin{array}{@{\,\,\,\,}c@{\,\,\,\,}c@{\,\,\,\,}c@{\,\,\,\,}c@{\,\,\,\,\,\,\,\,\,}c@{\,\,\,\,\,\,\,\,\,}c@{\,\,\,\,\,\,\,\,\,}c@{}}
&\tilde\sigma&\text{normal form}&\dim \Sigma&\Gas{}(0)\\
\hline
\\
&[2]&\pm \mathbb{I}_2&0&\PP_{\RR}^1
\vs
\\
&[1,1]&\text{-}\mathbb{I}_1\oplus \mathbb{I}_1&1&\{*\,*\}
\vs
\\
&[(1\bar 1)]&R_\theta&1&\emptyset
\vs
\\
\end{array}
$
\end{table}
\end{center}

\begin{center}
\begin{table}[h]
\caption{Isometries of $\Sp^2$}\label{tts2}
$
\begin{array}{@{\,\,\,\,}c@{\,\,\,\,}c@{\,\,\,\,}c@{\,\,\,\,}c@{\,\,\,\,\,\,\,\,\,}c@{\,\,\,\,\,\,\,\,\,}c@{\,\,\,\,\,\,\,\,\,}c@{}}
&\tilde\sigma&\text{normal form}&\dim \Sigma&\Gas{}(0)&\Gas{}(1)\\
\hline
\\
&[3]&\pm \mathbb{I}_3&0&\PP_{\RR}^2&\PP_{\RR}^2&
\vs
\\
&[2,1]&\pm(\text{-}\mathbb{I}_2\oplus \mathbb{I}_1)&2&\PP_{\RR}^1\sqcup\{*\}&\PP_{\RR}^1\sqcup\{*\}
\vs
\\
&[(1\bar 1),1]&R_\theta\oplus \pm\mathbb{I}_1&3&\{*\}&\{*\}
\vs
\\
\end{array}
$
\end{table}
\end{center}

\begin{center}
\begin{table}[h]
\caption{Isometries of $\Sp^3$}\label{tts3}
$
\begin{array}{@{}c@{\,\,\,}c@{\,\,\,\,\,\,\,}c@{\,\,\,\,\,\,}c@{\,\,\,\,\,\,}c@{\,\,\,\,\,\,\,}c@{\,\,\,\,\,\,}c@{}}
&\tilde\sigma&\text{normal form}&\dim\Sigma&\Gas{}(0)&\Gas{}(1)&\Gas{}(2)\\
\hline
\\
&[4]&\pm\mathbb{I}_4&0&\PP_{\RR}^3&\Grl{2}{\RR^4}&\PP_{\RR}^3
\vs
\\
&[(2\bar 2)]&R_{\theta}\oplus R_{\theta}&3&\emptyset&\PP^1_{\CC}&\emptyset
\vs
\\
&[3,1]&\pm(\text{-}\mathbb{I}_3\oplus\mathbb{I}_1)&3&\PP_{\RR}^2\sqcup\{*\}&\PP_{\RR}^2\sqcup\PP_{\RR}^2&\PP_{\RR}^2\sqcup\{*\}
\vs
\\
&[2,2]&\text{-}\mathbb{I}_2\oplus\mathbb{I}_2&4&\PP_{\RR}^1\sqcup\PP_{\RR}^1&(\PP_{\RR}^1\times\PP_{\RR}^1)\sqcup\{*\,*\}&\PP_{\RR}^1\sqcup\PP_{\RR}^1
\vs
\\
&[(1\bar 1),2]&R_{\theta}\oplus \pm\mathbb{I}_2&5&\PP_{\RR}^1&\{*\,*\}&\PP_{\RR}^1
\vs
\\
&[(1\bar 1),(1\bar 1)]&R_{\theta_1}\oplus R_{\theta_2}&6&\emptyset&\{*\,*\}&\emptyset
\vs
\\
&[(1\bar 1),1,1]&R_\theta\oplus \text{-}\mathbb{I}_1\oplus\mathbb{I}_1&6&\{*\,*\}&\{*\,*\}&\{*\,*\}
\vs
\\
\end{array}
$
\end{table}
\end{center}

\subsection{Invariant affine subspaces of $\EE^n$}
Recall that the standard Euclidean affine space is given by $\EE^n=\{x_{n+1}=1\}\subset V=\RR^{n+1}$, 
and that $V_0=\{x_{n+1}=0\}\subset V$ is its associated vector space. The canonical inclusion $V_0\to V$ induces
an inclusion of Grassmannian varieties $\Grl{k}{V_0}\to \Grl{k}{V}$. The bijection between $k$-dimensional
affine subspaces of $\EE^n$ and $(k+1)$-dimensional linear subspaces of $V$ not contained in $V_0$ induces an isomorphism of algebraic varieties
$\Gaf{k}{n}\cong\Grl{k+1}{V}\setminus \Grl{k+1}{V_0}.$
Since $\Gaf{k}{n}$ is an open Zariski connected subset of $\Grl{k+1}{V}$, it is a quasi-projective algebraic variety of dimension $(k+1)(n-k)$.
We will denote by $\pi:\Gaf{k}{n}\to \Grl{k}{V_0}$ 
the natural projection sending each affine subspace to its associated vector space.

Let $f\in \mathbf{I}(\EE^n)$ be an isometry, and let $\varphi_0:V_0\to V_0$ be its associated linear map. 
Recall that an affine subspace $p+V\subset\EE^n$ is $f$-invariant
if and only if $\varphi_0(V)\subset V$ and $f(p)-p\in V$. Therefore we have $\pi(\Gamma_f(k))\subset S_{\varphi_0}(k)$.
\begin{lem}
If $f$ and $g$ are conjugate Euclidean isometries, the varieties $\Gamma_f(k)$ and $\Gamma_g(k)$ are isomorphic for all $k$.
\end{lem}
\begin{proof}
Let $\alpha\in \Euc(n)$ such that $g=\alpha f\alpha^{-1}$. Then $p+V\mapsto \alpha(p+V)$ is an isomorphism.
\end{proof}
Let $f\in \mathbf{I}(\EE^n)$ be an isometry induced by $\varphi:V\to V$, and let $\varphi_0:V_0\to V_0$ be its associated linear map. Let
$V=V_R\oplus V_1$ be the decomposition of $V$ into $\varphi$-invariant subspaces, 
where $V_1$ denotes the generalized eigenspace of eigenvalue $1$ and $V_R$ is the direct 
sum of the remaining generalized eigenspaces. Denote by $\varphi_R=\varphi|_{V_R}$ and 
$\varphi_1=\varphi|_{V_1}$ the restrictions of $\varphi$ to $V_R$ and $V_1$ respectively.

Let $\BB=\EE^n/V_R$ be the quotient affine space with associated vector space $V_{01}:=V_0/V_R\cong V_0\cap V_1.$
Note that $\BB$ is the affine space defined by $\BB=\{x\in V_1; x_{n+1}=1\}$ and the restriction of $\varphi_0$ to $V_{01}$ 
is the identity transformation. Since the map $\varphi_1:V_1\to V_1$ is unipotent, it induces an isometry $f_1:\BB\to \BB$, 
which is either the identity or a translation.
\begin{prop}\label{produ}
Let $f\in\mathbf{I}(\EE^n)$. With the previous notation, $\Gamma_f\cong S_{\varphi_R}\times \Gamma_{f_1}$. In particular,
$$\Gamma_f(k)\cong \bigsqcup_{k_1+k_2=k} S_{\varphi_R}(k_1)\times \Gamma_{f_1}(k_2).$$
\end{prop}
\begin{proof}
We have an isomorphism $\Gamma_f\cong S_\varphi\setminus S_{\varphi_0}$. By Proposition $\ref{decoautos}$, 
there is an isomorphism $S_\varphi\cong S_{\varphi_R}\times S_{\varphi_1}$. Since $V_R\subset V_0$ we 
have $S_{\varphi_0}\cong S_{\varphi_R}\times S_{\varphi_{01}}$.
Therefore
$$\Gamma_f\cong \left(S_{\varphi_R}\times S_{\varphi_1}\right) 
\setminus \left(S_{\varphi_R}\times S_{\varphi_{01}}\right)\cong S_{\varphi_R}\times \left(S_{\varphi_1} \setminus S_{\varphi_{01}}\right)
\cong S_{\varphi_R}\times \Gamma_{f_1}.$$
\end{proof}
\begin{prop}\label{redufix}
Let $f\in \mathbf{I}(\EE^n)$ be a unipotent isometry.
\begin{enumerate}[i)]
 \item If $f$ is elliptic then $\Gamma_{f}(k)\cong \Gr{k}{\EE^n}$ for all $k\geq 0$.
\item If $f$ is hyperbolic then $\Gamma_{f}(0)=\emptyset$, and $\Gamma_{f}(k)\cong \Gr{k-1}{\EE^{n-1}}$, for all $k\geq 1$. 
\end{enumerate}
\end{prop}
\begin{proof}
If $f$ is elliptic it is the identity transformation and every subspace is $f$-invariant.

Assume that $f$ is hyperbolic. Let $(\{e^i\}_{i=1}^n;p)$ be an Euclidean reference of $\EE^n$ such that
the matrix of $f$ in this reference is a normal form $\mathbb{I}_{n-1}\oplus\left(\begin{smallmatrix}1&a\\0&1\end{smallmatrix}\right)$. 
For all $p\in \EE^n$ we have $f(p)-p=ae^n$.
Let $p+V$ be an $f$-invariant subspace of dimension $k$ and define $L:=Sp\{e^n\}$. Then $L\subset V$,
since the invariance of $p+V$ implies $f(p)-p=ae^n\subset V$.

Consider the orthogonal complement $L^\bot$ of $L$ in $V_0\cap V_1$. Then $\BB=p+L^\bot$ is an
Euclidean affine space of dimension $n-1$. Denote by $\pi:\EE^n\to \BB$ the orthogonal projection of $\EE^n$ along
$L$, and by $i:\BB\to \EE^n$ the canonical inclusion. Let $g=\pi\circ f\circ i:\BB\to \BB$.
Then $g$ is the identity transformation.
There is a commutative diagram
$$
\begin{CD}
\EE^n@>f>>\EE^n\\
@AAi A @VV\pi V\\
\BB@>g>>\BB\\
\end{CD}
$$
Since $L\subseteq V$, $\pi(p+ V)$ is $g$-invariant of dimension $k-1$. Then $(p+ V)\to \pi(p+ V)$ defines an isomorphism
$\Gamma_f(k)\to \Gamma_g(k-1)$. Hence if $k\geq 1$, $\Gamma_{f}(k)\cong \Gamma_{g}(k-1)= \Gr{k-1}{\EE^{n-1}}.$
\end{proof}

\begin{teo}\label{main2}
Let $f\in\mathbf{I}(\EE^n)$ be an Euclidean isometry.
\begin{enumerate}[i)]
\item Let $\sigma_f=[\tilde\sigma;r;t]$ be the Segre symbol of $f$. Let $\varphi_R\in \mathrm{O}(n-r)$ be an orthogonal map with Segre symbol
$\tilde\sigma$. Then
$$\Gamma_f(k)\cong\bigsqcup_{k_1+k_2=k}S_{\varphi_R}(k_1)\times \Gr{k_2-d}{\EE^{r-d}},\text{ where }
\left\{\begin{array}{lll}
d=0\text{ if }t=e\text{ (elliptic)}\\
d=1\text{ if }t=h\text{ (hyperbolic)}
\end{array}
\right..
$$
\item The dimensions of the connected components of $\Gamma_f(k)$, for $k\leq 3$ determine the Segre symbol of $f$.
\end{enumerate}
\end{teo}
\begin{proof}
i). We can assume that $f$ is a normal form. Let $f=\varphi_R\oplus f_1$ be the decomposition of $f$ into its orthogonal and
unipotent components. Then $\Gamma_f\cong S_{\varphi_R}\times \Gamma_{f_1}$ by Proposition $\ref{produ}$.
The result follows from Proposition $\ref{redufix}$.

ii). If $\Gamma_f(0)\neq\emptyset$ then $f$ is elliptic, while if $\Gamma_f(0)=\emptyset$ 
then $f$ is hyperbolic. We study each case separately.

Assume that $f$ is elliptic. Then $\Gamma_f(0)\cong \EE^r$ for some $r\geq 0$ and by
i), $\sigma_f=[\tilde\sigma;r;e]$. 
Let us see that the dimensions of the connected components of $\Gamma_f(1)$ and $\Gamma_f(2)$ determine $\tilde\sigma$.
Indeed, since $r$ is known, they determine the dimensions
of $S_{\varphi_R}(k)$, for $k=1,2$. By Proposition $\ref{lineaconnex}$, these determine $\tilde\sigma$.

Assume that $f$ is hyperbolic. Then for $k\geq 1$, $\Gamma_f(k)\cong \Gamma_g(k-1)$, where $\sigma_g=[\tilde\sigma;r;e]$. 
By the previous case, $\sigma_g$ is determined by the dimensions of $\Gamma_g(k)$, for $k\leq 2$. 
Therefore $\sigma_f$ is determined by the dimensions of $\Gamma_f(k)$, for $k\leq 3$.
\end{proof}

Tables $\ref{tte1}$ to $\ref{tte3}$ show a normal form representative of each Segre class for isometries 
of $\EE^1$, $\EE^2$ and $\EE^3$, together with the varieties of invariant affine subspaces of each dimension.
\begin{center}
\begin{table}[h]
\caption{Isometries of $\EE^1$}\label{tte1}
$
\begin{array}{@{\,\,\,\,}c@{\,\,\,\,}c@{\,\,\,\,}c@{\,\,\,\,}c@{\,\,\,\,\,\,\,\,\,}c@{\,\,\,\,\,\,\,\,\,}c@{\,\,\,\,\,\,\,\,\,}c@{}}
&\sigma&\text{normal form}&\dim\Sigma&\Gamma(0)\\
\hline
\\
&[0;1;e]&\tiny{\left(\begin{array}{@{\,}c@{\,\,\vline\,\,}c@{\,}}
1& \\
\hline
 &1
\end{array}\right)}&0&\EE^1
\vs
\\
&[1;0;e]&\tiny{\left(\begin{array}{@{\,}c@{\,\,\vline\,\,}c@{\,}}
\text{-}1& \\
\hline
 &1
\end{array}\right)}&1&\{*\} 
\vs
\\
&[0;1;h]&\tiny{\left(\begin{array}{@{\,}c@{\,\,\vline\,\,}c@{\,}}
1&a \\
\hline
 &1
\end{array}\right)}&1&\emptyset
\\
\vs
\\
\end{array}
$
\end{table}
\end{center}

\begin{center}
\begin{table}[h]
\caption{Isometries of $\EE^2$}\label{tte2}
$
\begin{array}{@{\,\,\,\,}c@{\,\,\,\,}c@{\,\,\,\,}c@{\,\,\,\,}c@{\,\,\,\,\,\,\,\,\,}c@{\,\,\,\,\,\,\,\,\,}c@{\,\,\,\,\,\,\,\,\,}c@{}c@{}c@{}}
&\sigma&\text{normal form}&\dim\Sigma&\Gamma(0)&\Gamma(1)\\
\hline
\\
&[0;2;e]&
\tiny{\left(
\begin{array}{@{}c@{\,\,}c@{\,\,\vline\,\,}c@{\,}}
1& & \\
 &1& \\
\hline
 & &1\\
\end{array}
\right)} &0&\EE^2&\Gaf{1}{2}
\vs\\
&[1;1;e]&
\tiny{\left(
\begin{array}{@{}r@{\,\,}r@{\,\,\vline\,\,}r@{\,}}
\text{-}1& & \\
 &1& \\
\hline
 & &1\\
\end{array}
\right)} &2&\EE^1&\EE^1\sqcup\{*\}
\vs\\
&[2;0;e]&
\tiny{\left(
\begin{array}{@{}r@{\,\,}r@{\,\,\vline\,\,}r@{\,}}
\text{-}1& & \\
 &\text{-}1& \\
\hline
 & &1\\
\end{array}
\right)} &2&\{*\}&\PP_{\RR}^1
\vs\\
 &[0;2;h]&\tiny{\left(
\begin{array}{@{}r@{\,\,}r@{\,\,\vline\,\,}r@{\,}}
1& & \\
 &1&a\\
\hline
 & &1\\
\end{array}
\right)} &2&\emptyset&\EE^1
\vs\\
&[(1\bar{1});0;e]&
\tiny{\left(
\begin{array}{@{}r@{\,\,}r@{\,\,\vline\,\,}r@{\,}}
c&s& \\
\text{-}s&c& \\
\hline
 & &1\\
\end{array}
\right)} &3&\{*\}&\emptyset
\vs\\
&[1;1;h]&
\tiny{\left(
\begin{array}{@{}r@{\,\,}r@{\,\,\vline\,\,}r@{\,}}
\text{-}1& & \\
 &1&a\\
\hline
 & &1\\
\end{array}
\right)}&3&\emptyset&\{*\}
\vs\\
\end{array}
$
\end{table}
\end{center}

\begin{center}
\begin{table}[h]
\caption{Isometries of $\EE^3$}\label{tte3}
$
\begin{array}{@{}c@{}c@{}c@{}c@{\,\,\,\,}c@{\,\,\,\,}c@{\,\,\,\,\,}c@{}}
&\sigma&\text{normal form}&\dim\Sigma&\Gamma(0)&\Gamma(1)&\Gamma(2)\\
\hline
\\
&[0;3;e]&\mathbb{I}_4&0&\EE^3&\Gaf{1}{3}&\Gaf{2}{3}
\vs
\\
&[1;2;e]&\text{-}\mathbb{I}_1\oplus\mathbb{I}_3&3&\EE^2&\Gaf{1}{2}\sqcup \EE^2&\Gaf{1}{2}\sqcup\{*\}
\vs
\\
& [3;0;e]&\text{-}\mathbb{I}_3\oplus\mathbb{I}_1&3&\{*\}&\PP_{\RR}^2&\PP_{\RR}^2
\vs
\\
&[0;3;h]&\mathbb{I}_2\oplus\left(\begin{smallmatrix}1&a\\0&1\end{smallmatrix}\right)&3&\emptyset&\EE^2&\Gaf{1}{2}
\vs
\\
&[2;1;e]&\text{-}\mathbb{I}_2\oplus \mathbb{I}_2 &4&\EE^1&(\PP_{\RR}^1\times \EE^1)\sqcup\{*\}&\PP_{\RR}^1\sqcup\EE^1
\vs
\\
&[(1\bar{1});1;e]&R_\theta\oplus \mathbb{I}_2&5&\EE^1&\{*\}&\EE^1
\vs
\\
&[1;2;h]&\text{-}\mathbb{I}_1\oplus\mathbb{I}_1\oplus\left(\begin{smallmatrix}1&a\\0&1 \end{smallmatrix}\right)&5&\emptyset&\EE^1&\EE^1\sqcup \{*\}
\vs
\\
&[2;1;h]&\text{-}\mathbb{I}_2\oplus\left(\begin{smallmatrix}1&a\\0&1 \end{smallmatrix}\right)&5&\emptyset&\{*\}&\PP_{\RR}^1
\vs
\\
&[(1\bar{1}),1;0;3]&R_\theta\oplus\text{-}\mathbb{I}_1\oplus\mathbb{I}_1&6&\{*\}&\{*\}&\{*\}
\vs
\\
&[(1\bar{1});1;h]&R_\theta\oplus\left(\begin{smallmatrix}1&a\\0&1 \end{smallmatrix}\right)&6&\emptyset&\{*\}&\emptyset
\end{array}
$
\end{table}
\end{center}

\subsection{Invariant totally geodesic hyperbolic subspaces of $\HH^n$}
The hyperbolic Grassmannian $\Gh{k}{n}$ is the set of all $k$-dimensional totally geodesic hyperbolic subspaces $\HH^k$ of $\HH^n$.
These are precisely the intersections of $\HH^n$ with $(k+1)$-dimensional time-like subspaces of $\RR^{n+1}$.
Therefore $\Gh{k}{n}$ is an open connected subset of $\Grl{k+1}{\RR^{n+1}}$, 
but not a Zariski open set. Hence it is a semi-algebraic manifold of dimension $(k+1)(n-k)$.

If $f\in \mathbf{I}(\HH^n)$ is an isometry, the elements of $\Gah{f}(k)$ are in bijection with the
$(k+1)$-dimensional $f$-invariant time-like subspaces of $\RR^{n,1}$.
\begin{lem}
If $f$ and $g$ are conjugate isometries of the hyperbolic space, the varieties $\Gamma_f(k)$ and $\Gamma_g(k)$ are isomorphic for all $k$.
\end{lem}
\begin{proof}
It follows from the fact that Lorentz isometries preserve time-like subspaces.
\end{proof}
\begin{prop}\label{decohy}
Let $T=T_t\oplus T_s\in \mathrm{O}^+(1,n)$ be a space-time decomposition. Then $$\Gah{T}(k)\cong \bigsqcup_{k_1+k_2=k}\Gah{T_t}(k_1)\times S_{T_s}(k_2).$$
\end{prop}
\begin{proof}
Let $U\subset \RR^{n+1}$ be a time-like $T$-invariant subspace of dimension $k>0$.
By Lemma $\ref{lemapenjat}$ we have $U=U_t\oplus U_s$, where $U_t\subset V_t$ is $T_t$-invariant and $U_s\subset V_s$ is $T_s$-invariant.
Moreover $U_s$ is space-like, so for $U$ to be time-like, $U_t$ must be time-like. Therefore $U_t\in \Gah{T_t}(k_1)$, for some $k_1\geq 0$,
and $U_s\in S_{T_s}(k_2)$, such that $k_1+k_2=k$.
\end{proof}

\begin{prop}\label{concretes}
Let $T\in \mathrm{O}^+(1,n)$ and let $r=\dim V_t$.
\begin{enumerate}[1)]
\item If $T$ is elliptic then
$\Gah{T_t}(k)\cong \Gh{k}{r-1}$ for all $0\leq k\leq r-1$.
\item If $T$ is parabolic then $\Gah{T_t}(k)=\emptyset$ for $k=0,1$ and 
$\Gah{T_t}(k)\cong \Gaf{k-2}{r-3}$ for $k\geq 2$.
\item If $T$ is hyperbolic then $\Gah{T_t}(0)=\emptyset$ and $\Gah{T_t}(1)=\{V_t\} $.
\end{enumerate}
\end{prop}
\begin{proof}
Case 1. If $T$ is elliptic then $T_t$ is the identity transformation. Therefore every subspace is $T_t$-invariant and the result follows.

Case 2. Assume that $T$ is parabolic, and let $\{e_i\}_{i=0}^{r-1}$ be a Lorentz basis of $V_t$ such that the matrix of $T_t$ is $\Theta\oplus \mathbb{I}_{r-3}$.
To compute $\Gah{T_t}(k)$ we study the $T_t$-invariant time-like linear subspaces of $V_t$. We first show that if  $U\subset V_t$ is
a time-like $T_t$-invariant linear subspace of $V_t$ then $Sp\{e_1,e_0+e_2\}\subset U$.

Let $\{v_i\}_{i=0}^{r-1}$ be the basis of $V_t$ defined by $v_2={1\over 2}(e_0+e_2)$ and $v_i=e_i$ for all $i\neq 2$. Then
$$
T_tv_0=v_0+v_1+v_2,\,\,
T_tv_1=v_1+2v_2\,\,\text{ and }\,\,
T_t v_i=v_i \,, \forall\, i>1.
$$
Let $U\subset V_t$ be a time-like $T_t$-invariant linear subspace of $V_t$,
and let $u=\sum_{i=0}^{r-1}a_iv_i\in U$ be a time-like vector of $U$. Then
$(T-I)u=a_0v_1+(a_0+2a_1)v_2$ and $(T-I)^2u=2a_0v_2.$
Since $Q(u)<0$, we have $a_0\neq 0$, and since $(T-I)^2u\in U$ it follows that $v_2\in U$.
Therefore since $(T-I)u\in U$ we have $v_1\in U$. Hence $Sp\{v_1,v_2\}\subset U$ as claimed.

Note that since $Sp\{v_1,v_2\}$ is not time-like, every time-like $T_t$-invariant subspace of $V_t$
has dimension of at least 3. Hence $\Gah{T_t}(k)=\emptyset$, for $k\leq 1$.

For $k\geq 2$, the set of $k$-dimensional subspaces $U\subseteq V_t$ such that $Sp\{ v_1,v_2\}\subseteq U$ is 
isomorphic to $\Grl{k-2}{V_t/Sp\{v_1,v_2\}}\cong\Grl{k-2}{\RR^{r-2}}$. Furthermore, every such subspace is $T_t$-invariant.
We next show that the set of those $T_t$-invariant subspaces that contain $Sp\{v_1,v_2\}$ and are not time-like, is isomorphic to $\Grl{k-2}{\RR^{r-3}}$.

Let $V_0=\{x_0=0\} \cap V_t\subset V_t$ and let $U\subset V_t$ such that $Sp\{v_1,v_2\}\subset U$. 
Then $Q|_{U}\geq 0$ if and only if $U\subset V_0$. Indeed, assume that $Q|_U\geq0$ and
let $u=\sum_{i=0}^{r-1}a_iv_i\in U$. Then $Q(u)=-a_0(a_0+2a_2)+\sum _{i\neq2} a_i^2$. 
Therefore if $a_0\neq 0$, and since $v_2\in U$, we can choose $a_2$ such that $Q(u)<0$, 
which is a contradiction. Hence $a_0=0$ and $U\subset V_0$. Conversely since $V_0$ is 
space-like, if $U\subset V_0$ then $Q|_{U}\geq 0$.

Therefore the set of $T_t$-invariant $k$-dimensional time-like subspaces of $V_t$ is
isomorphic to $\Grl{k-2}{\RR^{r-2}}\setminus \Grl{k-2}{\RR^{r-3}}$, and we have
$$\Gah{T_t}(k)\cong\Grl{k-1}{\RR^{r-2}}\setminus \Grl{k-1}{\RR^{r-3}}\cong\Gr{k-2}{\EE^{r-3}}.$$

Case 3. If $T$ is hyperbolic then $r=2$ and $T_t$ has a normal form $\Omega=\left(\begin{smallmatrix}
c&d\\
d&c
\end{smallmatrix}
\right)$, where $c,d\in \RR$ are such that $c^2-d^2=1$, $d\neq 0$.
It follows from an easy computation that the only $T_t$-invariant proper
subspaces of $V_t$ are  light-like lines. Therefore $\Gah{T_t}(0)=\emptyset$ and $\Gah{T_t}(1)=\{V_t\}$.
\end{proof}

\begin{teo}\label{main3}
Let $T\in \mathrm{O}^+(1,n)$, and let $r=\dim V_t$.
\begin{enumerate}[i)]
 \item \begin{enumerate}[]
 \item If $T$ is elliptic then
$$\Gah{T}(k)\cong \bigsqcup_{k_1+k_2=k} \Gh{k_1}{r-1}\times S_{T_s}(k_2).$$
\item If $T$ is parabolic then
$\Gah{T}(k)=\emptyset$ for $k=0,1$ and
$$\Gah{T}(k)\cong \bigsqcup_{k_1+k_2=k-2} \Gaf{k_1}{r-3}\times S_{T_s}(k_2).$$
\item If $T$ is hyperbolic then
$\Gah{T}(0)=\emptyset$ and $\Gah{T}(k)\cong S_{T_s}(k-1)$ for $k\geq 1$.
\end{enumerate}
\item 
The dimensions of the connected components of $\Gah{T}(k)$,
for $k\leq 4$ determine the Segre symbol of $T$.
\end{enumerate}
\end{teo}
\begin{proof}
The first statement follows from Propositions $\ref{decohy}$ and $\ref{concretes}$. Let us prove the second.

Let $c_i$ denote the number of connected components of $\Gah{T}(i)$.
By the above proposition applied to the case $k=0$ we know that if $c_0>0$ then $T$ is elliptic, 
if $c_0=0$ and $c_1>0$ then $T$ is hyperbolic, and if $c_0=c_1=0$ then $T$ is parabolic.
Therefore we can study each type of isometry separately. Let $r=\dim V_t$. 
\begin{enumerate}[i)]
 \item 
If $T$ is elliptic then $\Gah{T}(0)\cong\Gh{0}{r-1}\cong \HH^{r-1}.$ Therefore the dimension of $\Gah{T}(0)$ determines $r$. Moreover,
$\Gah{T}(1)\cong \Gh{1}{r-1}\sqcup \left(\HH^{r-1}\times S_{T_s}(1)\right).$
Since $r$ is known, the dimensions of $\Gah{T}(1)$ determine the dimensions of $S_{T_s}(1)$. Therefore
$$\Gah{T}(2)\cong \Gh{2}{r-1}\sqcup\left(\Gh{1}{r-1}\times S_{T_s}(1)\right)\sqcup \left(\HH^{r-1}\times S_{T_s}(2)\right),$$
and the dimensions of $\Gah{T}(2)$ determine those of $S_{T_s}(2)$. By Proposition $\ref{lineaconnex}$ the Segre symbol of $T_s$ is
determined, so the Segre symbol of $T$ is determined as well.
\item If $T$ is hyperbolic then $\Gah{T}(k)\cong S_{T_s}(k-1)$ and the result follows from Proposition $\ref{lineaconnex}$.
\item 
If $T$ is parabolic then $\Gah{T}(2)\cong \HH^{r-2}$, so its dimension determines $r$. Moreover,
$$\Gah{T}(3)\cong \Gh{1}{r-2}\sqcup \left(\HH^{r-2}\times S_{T_s}(1)\right),$$
so the dimensions of $\Gah{T}(3)$ determine the dimensions of $S_{T_s}(1)$. Therefore
$$\Gah{T}(4)\cong \Gh{2}{r-2}\sqcup \left(\Gh{1}{r-2}\times S_{T_s}(1)\right)\sqcup \left(\HH^{r-2}\times S_{T_s}(2)\right),$$
so the dimensions of $\Gah{T}(4)$ determine the ones of $S_{T_s}(2)$. By Proposition $\ref{lineaconnex}$ we get the result.
\end{enumerate}
\end{proof}
Tables $\ref{tth1}$ to $\ref{tth3}$ show a normal form representative of each Segre class, 
for isometries of $\HH^1$, $\HH^2$ and $\HH^3$, together with the varieties of invariant hyperbolic subspaces of each dimension.

\begin{center}
\begin{table}[h]
\caption{Isometries of $\HH^1$}\label{tth1}
$
\begin{array}{@{\,\,\,\,}c@{\,\,\,\,}c@{\,\,\,\,}c@{\,\,\,\,}c@{\,\,\,\,\,\,\,\,\,}c@{\,\,\,\,\,\,\,\,\,}c@{\,\,\,\,\,\,\,\,\,}c@{}}
&\sigma&\text{normal form}&\dim\Sigma&\Gah{}(0)\\
\hline
\\
&[2;0;e]&\tiny{\left(\begin{smallmatrix}1\\&1\end{smallmatrix}\right)}&0&\HH^1
\vs
\\
&[1;1;e]&\tiny{\left(\begin{smallmatrix}1\\&\text{-}1\end{smallmatrix}\right)}&1&\{*\}
\vs
\\
&[2;0;h]&\tiny{\left(\begin{smallmatrix}c&d\\d&c\end{smallmatrix}\right)}&1&\emptyset
\end{array}
$
\end{table}
\end{center}

\begin{center}
\begin{table}[h]
\caption{Isometries of $\HH^2$}\label{tth2}
$
\begin{array}{@{\,\,\,\,}c@{\,\,\,\,}c@{\,\,\,\,}c@{\,\,\,\,}c@{\,\,\,\,\,\,\,\,\,}c@{\,\,\,\,\,\,\,\,\,}c@{\,\,\,\,\,\,\,\,\,}c@{}}
&\sigma&\text{normal form}&\dim\Sigma&\Gah{}(0)&\Gah{}(1)\\
\hline
\\
&[3;0;e]&\tiny{\left(\begin{smallmatrix}1\\&1\\&&1\end{smallmatrix}\right)}&0&\HH^2&\Gh{1}{2}
\vs
\\
&[2;1;e]&\tiny{\left(\begin{smallmatrix}1\\&1\\&&\text{-}1\end{smallmatrix}\right)}&2&\HH^1&\HH^1\sqcup\{*\}
\vs
\\
&[1;2;e]&\tiny{\left(\begin{smallmatrix}1\\&\text{-}1\\&&\text{-}1\end{smallmatrix}\right)}&2&\{*\}&\PP_{\RR}^1
\vs
\\
&[3;0;p]&\tiny{\left(\begin{smallmatrix}{3\over2}&1&\text{-}{1\over2}\\1&1&\text{-}1\\{1\over2}&1&{1\over2}\end{smallmatrix}\right)}&2&\emptyset&\emptyset&
\vs
\\
&[1;(1\bar 1);e]&\tiny{\left(\begin{smallmatrix}1\\&a&b\\&\text{-}b&a\end{smallmatrix}\right)}&3&\{*\}&\emptyset
\vs
\\
&[2;1;h]&\tiny{\left(\begin{smallmatrix}c&d\\d&c\\&&\pm 1\end{smallmatrix}\right)}&3&\emptyset&\{*\}
\vs
\\
\end{array}
$
\end{table}
\end{center}

\begin{center}
\begin{table}[h]
\caption{Isometries of $\HH^3$}\label{tth3}
$
\begin{array}{@{}c@{}c@{\,\,\,\,\,\,}c@{\,\,\,\,\,\,}c@{\,\,\,\,\,\,\,\,\,}c@{\,\,\,\,\,\,}c@{\,\,\,\,\,\,}c@{}}
&\sigma&\text{normal form}&\dim\Sigma&\Gah{}(0)&\Gah{}(1)&\Gah{}(2)\\
\hline
\\
&[4;0;e]&\mathbb{I}_4&0&\HH^3&\Gr{1}{\HH^3}&\Gr{2}{\HH^3}
\vs
\\
&[3;1;e]&\mathbb{I}_3\oplus\text{-}\mathbb{I}_1&3&\HH^2&\Gr{1}{\HH^2}\sqcup\HH^2& \Gr{1}{\HH^2}\sqcup\{*\}
\vs
\\
&[1;3;e]&\mathbb{I}_1\oplus \text{-}\mathbb{I}_3&3&\{*\}&\PP_{\RR}^2&\PP_{\RR}^2
\vs
\\
&[2;2;e]&\mathbb{I}_2\oplus\text{-}\mathbb{I}_2&4&\HH^1&(\PP_{\RR}^1\times \HH^1)\sqcup\{*\}&\HH^1\sqcup\PP_{\RR}^1
\vs
\\
&[4;0;p]&\Theta\oplus\mathbb{I}_1&4&\emptyset&\emptyset&\EE^1
\vs
\\
&[2;(1\bar 1);e]&\mathbb{I}_2\oplus R_\theta&5&\HH^1&\{*\}&\HH^1
\vs
\\
&[3;1;p]&\Theta\oplus\text{-}\mathbb{I}_1&5&\emptyset&\emptyset&\{*\}
\vs
\\
&[2;2;h]&\Omega_t\oplus\pm\mathbb{I}_2&5&\emptyset&\{*\}&\PP_{\RR}^1
\vs
\\
&[1;(1\bar 1),1;e]&\mathbb{I}_1\oplus R_\theta\oplus \text{-}\mathbb{I}_1&6&\{*\}&\{*\}&\{*\}
\vs
\\
&[2;1,1;h]&\Omega_t\oplus\mathbb{I}_1\oplus\text{-}\mathbb{I}_1&6&\emptyset&\{*\}&\{*\,*\}
\vs
\\
&[2;(1\bar1);h]&\Omega_t\oplus R_\theta&6&\emptyset&\{*\}&\emptyset
\end{array}
$\end{table}
\end{center}

\newpage

\section*{Acknowledgments}
I thank V. Navarro for pointing me to the particular problem that
gave rise to this paper and F. Guill\'{e}n for his valuable comments and suggestions.

\linespread{1}
\bibliographystyle{amsalpha}
\bibliography{bibliografia}
\mbox{}\\
\linespread{1.2}

\end{document}